\documentclass[10pt]{amsart}
\usepackage{epsfig}

\usepackage[T1]{fontenc}    
\usepackage{ae,aecompl}     
\usepackage{amsmath}        
\usepackage{amssymb}
\usepackage{amsthm}
\usepackage{amsfonts}
\usepackage{amscd}
\usepackage{color}
\usepackage{mathrsfs}       
\usepackage{bbm}            

\usepackage{bm}             
\usepackage{url}            
\usepackage{paralist}       
\usepackage{graphics}       
\usepackage{epic}          
\usepackage{overpic}   
\usepackage{xspace}         

\newtheorem{thm}{Theorem}[section]
\newtheorem{cor}[thm]{Corollary}
\newtheorem{lem}[thm]{Lemma}
\newtheorem{prop}[thm]{Proposition}
\newtheorem{rem}[thm]{Remark}

\newtheorem{expl}[thm]{Example}
\newtheorem{conj}[thm]{Conjecture}

\def\Ass{{\mathsf{Asso}}}  
\def\Perm{{\mathsf{Perm}}}
\def\cl{{\textrm{cl}}}    
\def\R{{\mathbb R}}       
\def\Z{{\mathbb Z}}       
\def\H{{\mathscr H}}
\def\A{{\mathcal A}}
\def\F{{\mathcal F}}
\def\a{{\alpha}}

\def\pf{\mathcal G}
\def\G{\mathcal G}

\def\ww{\bf w_0}
\def\apr{{\geq -1}}       

\def\lr{\mathsf {Lr}}
\def\pid{\pi^c_\downarrow}

\newcommand{\scalprod}[2]{{\langle #1,#2\rangle}}                         
\newcommand{\set}[2]{\left\{#1\vphantom{#2}\right.\;\left|\;\vphantom{#1}#2\right\}}

\def\ra{{\rangle}}
\def\la{{\langle}}

\begin{document}


\title{Permutahedra and  Generalized Associahedra}

\author{Christophe  Hohlweg}
\address[Christophe Hohlweg]{Universit\'e du Qu\'ebec \`a Montr\'eal\\
LaCIM et D\'epartement de Math\'ematiques\\ C.P. 8888 Succ. Centre-Ville\\
Montr\'eal, Qu\'ebec, H3C 3P8\\ CANADA}
\email{hohlweg.christophe@uqam.ca}
\urladdr{http://www.lacim.uqam.ca/\~{}hohlweg}

\author{Carsten Lange}
\address[Carsten Lange]{Freie Universit{\"a}t Berlin\\
Institut f{\"u}r Mathematik und Informatik\\
Arnimallee 3\\
14195 Berlin\\
Germany}
\email{lange@math.tu-berlin.de}

\author{Hugh Thomas}
\address[Hugh Thomas]{Department of Mathematics and
Statistics\\University of New Brunswick\\
Fredericton, New Brunswick, E3B 5A3\\
CANADA}
\email{hugh@math.unb.ca}

\date{April 16, 2008}

\thanks{Hohlweg was partially supported by Canadian Research Chair and the Fields Institute.
Hohlweg is partly supported by a startup grant from PAFARC.  Thomas is partly
supported by a Discovery Grant from NSERC}

\begin{abstract}
\noindent Given a finite Coxeter system~$(W,S)$ and a Coxeter
element~$c$, or equivalently an orientation of the Coxeter graph of
$W$, we construct a simple polytope whose outer normal fan is
N.~Reading's Cambrian fan~$\mathcal F_c$, settling a conjecture of
Reading that this is possible. We call this polytope the
$c$-generalized associahedron. Our approach generalizes Loday's
realization of the associahedron (a type~$A$ $c$-generalized
associahedron whose outer normal fan is not the cluster fan but a
coarsening of the Coxeter fan arising from the Tamari lattice) to
any finite Coxeter group. A crucial role in the construction is
played by the $c$-singleton cones, the cones in the $c$-Cambrian fan
which consist of a single maximal cone from the Coxeter fan.

Moreover, if~$W$ is a Weyl group and the vertices of the permutahedron
are chosen in a lattice associated to~$W$, then we show that our
realizations have integer coordinates in this lattice.
\end{abstract}

\maketitle

\tableofcontents

\section*{Introduction}

Let $(W,S)$ be a finite Coxeter system acting by reflections on an
$\mathbb R$-Euclidean space. Let~$\pmb a$ be a point in the
complement of the hyperplanes corresponding to the reflections in~$W$.
The convex hull of the $W$-orbit of~$\pmb a$ is a simple
convex polytope: the well-known {\em permutahedron} $\Perm^{\pmb a}(W)$.
The normal fan of $\Perm^{\pmb a}(W)$ is the {\em Coxeter fan}~$\F$.

{\em Cluster fans} were introduced by S.~Fomin and A.~Zelevinsky
in their work on cluster algebras~\cite{fomin_zelevinsky}. To each 
Weyl group there corresponds a cluster fan which encodes important 
algebraic information including the exchange graph of the corresponding 
cluster algebra. One very natural question is to find realizations 
of these fans as normal fans of simple polytopes. This was first 
answered by F.~Chapoton, S.~Fomin, and A.~Zelevinsky 
in~\cite{chapoton_fomin_zelevinsky}: for each Weyl group~$W$, they 
construct a simple convex polytope whose normal fan is the cluster 
fan. Such a polytope is called a {\em generalized associahedron of type~$W$}.

In~\cite{reading1}, N.~Reading introduced {\em Cambrian lattices} which
generalize the Tamari lattice to any finite Coxeter group~$W$ and Coxeter
element~$c \in W$. The $c$-Cambrian lattice is defined as a quotient of the
weak order on~$W$, and therefore has an associated fan, the
{\em $c$-Cambrian fan}~$\F_c$, which is a coarsening of the Coxeter fan~$\F$.
N.~Reading conjectured that any $c$-Cambrian fan is the normal
fan of a simple convex polytope~\cite[Conjecture 1.1]{reading1}.

We will call a polytope with the $c$-Cambrian fan as normal fan a
{\em $c$-generalized associahedron}~$\Ass_c(W)$.  In~\cite{reading4}, 
N.~Reading and D.~Speyer showed that~$\F_c$ is a simplicial fan, so 
$c$-generalized associahedra are simple polytopes. In the case 
where~$W$ is a Weyl group, they also showed that $c$-Cambrian fans 
are all combinatorially isomorphic to the cluster fan. Moreover, 
for {\em bipartite}~$c$, they show that~$\F_c$ is linearly isomorphic 
to the cluster fan. In other words,~$\Ass_c(W)$ has the same
combinatorial type as a generalized associahedron of type~$W$,
and for bipartite~$c$,~$\Ass_c(W)$ is linearly isomorphic to a
generalized associahedron of type~$W$.

In this article, we construct $c$-generalized associahedra for all~$c$, 
settling Reading's conjecture.

\smallskip
Let us review the history of some realizations. For
symmetric groups, that is for Coxeter groups of type A,
$c$-generalized associahedra are combinatorially isomorphic to the
{\it classical associahedron} $\Ass(S_n)$, whose combinatorial
structure was first described by J.~Stasheff in
1963~\cite{stasheff}. The combinatorial structure in question is
the face lattice of a simple convex polytope whose $1$-skeleton is
isomorphic to the undirected Hasse diagram of the Tamari lattice
on the set~$Y_{n}$ of planar binary trees with $(n+1)$-leaves (see
for instance~\cite{lee}) and therefore a fundamental example of a
secondary polytope as described
in~\cite{Gelfend_Kapranov_Zelevinsky}. Numerous realizations of
the associahedron have been given,
see~\cite{chapoton_fomin_zelevinsky,loday} and the references
therein. An elegant and simple realization of~$\Ass(S_n)$ by
picking some of the inequalities for the
permutahedron~$\Perm(S_n)$ is due to S.~Shnider
\&~S.~Sternberg~\cite{shnider_sternberg} (for a corrected version
consider J.~Stasheff \&~S.~Shnider~\cite[Appendix~B]{stasheff2}).
This realization implies a surjective map from the vertices of the
permutahedron to the vertices of the associahedron, which turns
out to be the well-known surjection from~$S_{n}$ to~$Y_{n}$ that
maps the weak order on~$S_n$ to the Tamari lattice as
described in~\cite[Sec.~9]{bjoerner_wachs}. Recently, J.-L.~Loday
presented an algorithm to compute the vertex coordinates of this
realization,~\cite{loday}: label the vertices of the associahedron
by planar binary trees with~$n+2$ leaves and apply a simple
algorithm on trees to obtain integer coordinates in~$\R^{n}$. This
polytopal realization of the associahedron from the permutahedron
is referred to as {\em Loday's realization}. An important aspect
of this construction is the fact that the Tamari lattice is a
quotient lattice, as well as a sublattice, of the weak order
on~$S_n$.

\begin{figure}[b]
      \begin{center}
      \begin{minipage}{0.95\linewidth}
         \begin{center}
         \begin{overpic}
            [width=0.9\linewidth]{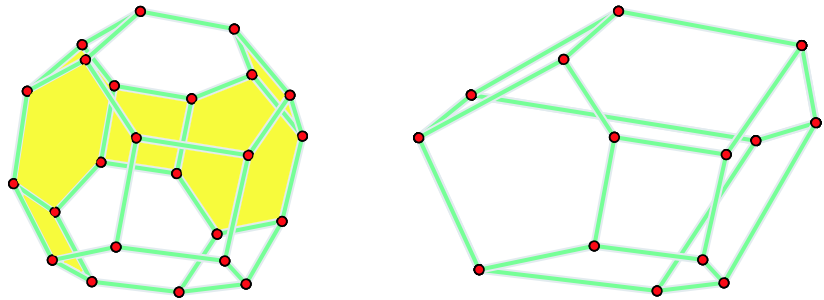}
         \end{overpic}
         \end{center}
         \caption[]{Obtaining the associahedron from the permutahedron for the Coxeter group~$S_4$
                    and Coxeter element~$c=s_1s_2s_3$.  The left picture shows the permutahedron
                    with the facets contained in the boundary of $c$-admissible half spaces translucent
                    and the facets contained in the boundary of non $c$-admissible half spaces shaded.
                    The picture to the right shows the associahedron obtained from the
                    permutahedron after removal of all non $c$-admissible half spaces. }
         \label{fig:remove_facets_c=_123}
      \end{minipage}
      \end{center}
\end{figure}

\begin{figure}[t]
      \begin{center}
      \begin{minipage}{0.95\linewidth}
         \begin{center}
         \begin{overpic}
            [width=0.9\linewidth]{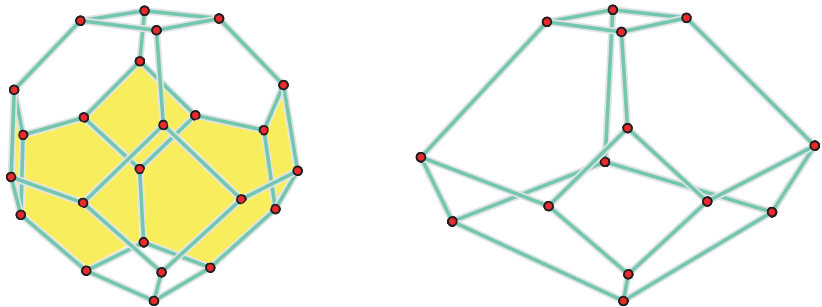}
         \end{overpic}
         \end{center}
         \caption[]{Obtaining the associahedron from the permutahedron for the Coxeter group~$S_4$
                    and Coxeter element~$c=s_2s_1s_3$. The left picture shows the permutahedron
                    with the facets contained in the boundary of $c$-admissible half spaces translucent
                    and the facets contained in the boundary of non $c$-admissible half spaces shaded.
                    The picture to the right shows the associahedron obtained from the
                    permutahedron after removal of all non $c$-admissible half spaces.}
         \label{fig:remove_facets_c=_213}
      \end{minipage}
      \end{center}
\end{figure}

For Coxeter groups of types~$A$ and~$B$, the first two authors
gave recently a Loday-type realization of any $c$-generalized
associahedron~\cite{HL}, and they showed that Loday's realization
of the associahedron is a $c$-generalized associahedron for a 
particular~$c$.  For hyperoctahedral groups, that
is Coxeter groups of type~$B$, the generalized associahedron is
called a {\em cyclohedron}.  It was first described by R.~Bott and
C.~Taubes in 1994~\cite{bott_taubes} in connection with knot theory,
and rediscovered independently by R.~Simion~\cite{simion}. See
also~\cite{chapoton_fomin_zelevinsky,markl,postnikov,reiner,simion};
none of these realizations is similar to Loday's (type $A$)
realization.

\smallskip
The realization of~$\Ass_c(W)$ for any finite Coxeter group~$W$ and
any Coxeter element~$c$ given in this article generalizes to any
finite Coxeter group the approach initiated by Shnider~\&~Sternberg
in type $A$, and extended by the two first authors to types~$A$ and
$B$.

\smallskip

Our construction of the $c$-generalized associahedron is very straightforward.
Start from~$\Perm^{\pmb a}(W)$ and its H-representation by non-redundant half 
spaces and fix a particular reduced expression of the longest element~$w_0$ of~$W$
determined by~$c$. The prefixes of~$w_0$ up to commutation of commuting simple reflections
will be called {\em $c$-singletons}. A half space~$\H$ of the H-representation
of~$\Perm^{\pmb a}(W)$ is {\em $c$-admissible} if its boundary contains a vertex
of~$\Perm^{\pmb a}(W)$ that corresponds to a $c$-singleton. The intersection of all
$c$-admissible half spaces is a $c$-generalized associahedron~$\Ass^{\pmb a}_c(W)$
and its normal fan is the $c$-Cambrian fan $\F_c$ (Corollary~\ref{cor:Intro}).
The selection of $c$-admissible and the removal of not $c$-admissible half spaces for
the permutahedron of type~$A$ and distinct choices for the Coxeter element is illustrated
in Figure~\ref{fig:remove_facets_c=_123} and Figure~\ref{fig:remove_facets_c=_213}.

Moreover, if $W$ is a Weyl group and the vertices of the permutahedron~$\Perm^{\pmb a}(W)$ 
are chosen in a suitable lattice associated to $W$, then we show that $\Ass_c^{\pmb a}(W)$ 
has integer coordinates in this lattice (Theorem~\ref{thm:IntegerCoordinates}).

\smallskip
Another interesting aspect of this construction is that we are able to recover the
$c$-cluster complex: relating cluster fans to quiver theory, R.~Marsh, M.~Reineke 
and A.~Zelevinsky introduce in~\cite{MRZ} what N.~Reading and D.~Speyer call the 
{\em $c$-cluster fan} in~\cite{reading4}, and its associated simplicial complex
{\em the $c$-cluster complex}. A $c$-cluster fan is a generalization of the cluster 
fan to any finite Coxeter group~$W$ and Coxeter element $c\in W$ ($c$ bipartite is 
then the traditional case); its applications in quiver representations are most 
natural for~$W$ of types~$A$,~$D$ and~$E$.

By replacing the natural labeling of the maximal faces
of~$\Ass_c^{\pmb a}(W)$ by a labeling that uses almost positive
roots only, we obtain the $c$-cluster complex. This replacement is
determined by an easy combinatorial rule as stated in
Theorem~\ref{thm:Explif_c}. This suggests that these constructions
will play an important role in the study of $c$-cluster complexes
and related structures.

\smallskip
This article is organized as follows. In~\S\ref{se:CSing}, we recall
some facts about finite Coxeter groups, Coxeter sortable elements,
and Cambrian lattices. Additionally, the important notion of a
$c$-singleton is defined and fundamental properties are proven. 
In~\S\ref{se:CambFan}, we recall some facts about fans, in particular
of Coxeter and Cambrian fans, and give a precise combinatorial
description of the rays of Cambrian fans. In~\S\ref{se:CamPoly}, we
state and prove our main result (Theorem~\ref{thm:main_theorem}).
Finally, in~\S\ref{se:Remaarks} we study some specific examples of
finite reflection groups. We work out the dihedral case explicitly
to show that the vertex barycentres of the permutahedra and
associahedra coincide and we explain how the realizations given 
in~\cite{HL} for type~$A$ and~$B$ are a particular instances of the
construction described in this article.

In a sequel~\cite{BHLT}, we describe the isometry classes of these
realizations.

The construction presented in this article has been implemented as
the set of functions {\tt CAMBRIAN} to be used with the library {\tt
CHEVIE} for {\tt GAP}~\cite{gap,chevie} and can be found on the
first author's web page,~\cite{hohlweg}.

\section{Coxeter-singletons and Cambrian lattices}\label{se:CSing}

Let $(W,S)$ be a finite Coxeter system. We denote by~$e$ the identity of~$W$ and
by $\ell: W\to \mathbb N$ the length function on~$W$. Let $n=|S|$ be the rank of~$W$.
Denote by $w_0$ the unique element of maximal length in~$W$.

The {\em (right) weak order} $\leq$ on $W$ can be defined by~$u\leq v$ if and only
if there is a $v'\in W$ such that $v=uv'$ and $\ell(v)=\ell(u)+\ell(v')$. The
{\em descent set} $D(w)$ of~$w\in W$ is $\{s\in S\,|\, \ell(ws)<\ell(w)\}$. A
{\em cover} of~$w\in W$ is an element~$ws$ such that $s\notin D(w)$.

The subgroup~$W_I$ generated by~$I\subseteq S$ is a {\em (standard) parabolic subgroup}
of~$W$ and the set of minimal length (left) coset representatives of~$W/W_I$ is given by
\[
  W^I =\{x\in W \,|\, \ell(xs)>\ell(x),\,\forall s\in I\}
      = \{ x \in W \,|\, D(x) \subseteq S\setminus I \}.
\]
Moreover, each~$w \in W$ has a unique decomposition~$w = w^I w_I$ where~$w^I \in W^I$
and~$w_I \in W_I$. Moreover, $\ell(w) = \ell(w^I) + \ell(w_I)$, see~\cite[\S 5.12]{humphreys}.
The pair~$(w^I, w_I)$ is often called the {\em parabolic components} of~$w$ along~$I$.
For~$s\in S$ we follow N.~Reading's notation and set $\la s\ra:=S\setminus\{s\}$.

Let~$c$ be a Coxeter element of~$W$, that is, the product of the simple reflections of~$W$
taken in some order, and fix a reduced expression for $c$.

\subsection{$c$-sortable elements}

For $I\subset S$, we denote by~$c_{(I)}$ the subword of~$c$ obtained by considering
only simple reflections in~ $I$. Obviously, $c_{(I)}$ is a Coxeter element of~$W_I$.
For instance, take~$W=S_5$ and $S=\{s_i\,|\,1\leq i\leq 4\}$ where~$s_i$ denotes the
simple transposition~$(i,i+1)$. If $c=s_1s_3s_4s_2$ and $I=\{s_2,s_3\}$
then~$c_{(I)}=s_3s_2$. Consider the possible ways to write~$w\in W$ as a reduced
subword of the infinite word $c^\infty=cccccc\dots$. In~\cite[\S2]{reading2},
N.~Reading defines the {\em $c$-sorting word} of~$w\in W$ as the reduced subword
of~$c^\infty$ for~$w$ which is lexicographically first as a sequence of positions.
The~$c$-sorting word of~$w$ can be written as $c_{(K_1)}c_{(K_2)}\dots c_{(K_p)}$
where~$p$ is minimal for the property:
\[
 w=c_{(K_1)}c_{(K_2)}\dots c_{(K_p)}\quad \textrm{and} \quad\ell(w)=\sum_{i=1}^p |K_i|.
\]
\begin{figure}
      \begin{center}
      \begin{minipage}{0.95\linewidth}
         \begin{center}
         \begin{overpic}
            [width=0.9\linewidth]{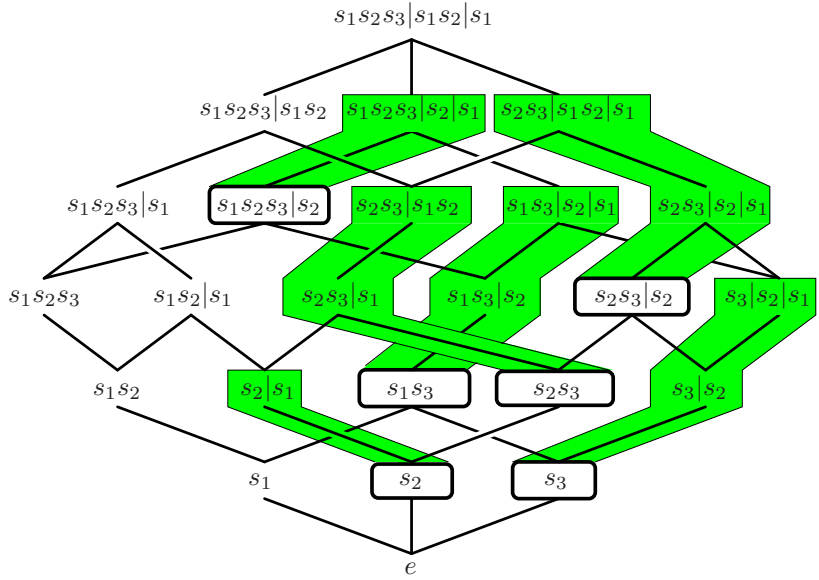}
            \put(40.5,70){$s_1s_2s_3|s_1s_2|s_1$}
            \put(24,58.5){$s_1s_2s_3|s_1s_2$}
            \put(42,58.5){$s_1s_2s_3|s_2|s_1$}
            \put(61,58.5){$s_2s_3|s_1s_2|s_1$}
            \put(7.7,46.7){$s_1s_2s_3|s_1$}
            \put(26.1,46.7){$s_1s_2s_3|s_2$}
            \put(43.2,46.7){$s_2s_3|s_1s_2$}
            \put(61.7,46.7){$s_1s_3|s_2|s_1$}
            \put(80.3,46.7){$s_2s_3|s_2|s_1$}
            \put(0.5,35.5){$s_1s_2s_3$}
            \put(18.3,35.5){$s_1s_2|s_1$}
            \put(36.5,35.5){$s_2s_3|s_1$}
            \put(54.4,35.5){$s_1s_3|s_2$}
            \put(72.5,35.5){$s_2s_3|s_2$}
            \put(88.5,35.5){$s_3|s_2|s_1$}
            \put(10.9,24){$s_1s_2$}
            \put(29,24){$s_2|s_1$}
            \put(47,24){$s_1s_3$}
            \put(65,24){$s_2s_3$}
            \put(82,24){$s_3|s_2$}
            \put(30,13){$s_1$}
            \put(48.5,13){$s_2$}
            \put(66.3,13){$s_3$}
            \put(49.2,2){$e$}
         \end{overpic}
         \end{center}
         \caption[]{$c=s_1s_2s_3$}
         \label{fig:cambrian_lattice_projection_123}
      \end{minipage}
      \end{center}
\end{figure}

The sequence~$c_{(K_1)}, \ldots,c_{(K_p)}$ associated to the $c$-sorting word
for~$w$ is called {\em the $c$-factorization} of~$w$. The $c$-factorization
of~$w$ is independent of the chosen reduced word for~$c$ but depends on the Coxeter
element~$c$. In general the $c$-factorization does not yield a nested sequence
$K_1, \ldots, K_p$ of subsets of~$S$. An element~$w \in W$ is called
{\em $c$-sortable} if $K_1 \supseteq K_2 \supseteq \ldots \supseteq K_p$. It
is clear that for any chosen Coxeter element~$c$, the identity~$e$ is $c$-sortable,
and Reading proves in~\cite{reading2} that the longest element $w_0 \in W$ is
$c$-sortable as well. The $c$-factorization of $w_0$ is of particular importance
for us and is denoted by~$\ww$. To illustrate these notions, consider~$W=S_4$ with
generators~$S=\{s_1,s_2,s_3\}$. The weak order of~$S_4$ with elements represented
by their $c$-factorization is shown in Figure~\ref{fig:cambrian_lattice_projection_123}
for $c=s_1s_2s_3$ and Figure~\ref{fig:cambrian_lattice_projection_213} for $c=s_2s_1s_3$
(the delimiter~`$|$' indicates the end of $K_i$ and the beginning of $K_{i+1}$).
Moreover, the background colour carries additional information: The background of~$w$
is white if and only if~$w$ is $c$-sortable.

\subsection{$c$-Cambrian lattice}
N.~Reading shows that the $c$-sortable elements constitute a sublattice of the weak
order of~$W$ which is called the {\em $c$-Cambrian lattice},~\cite{reading1,reading3}.
A Cambrian lattice is also a lattice quotient of the weak order on~$W$. In particular,
there is a downward projection~$\pi^c_\downarrow$ from~$W$ to the $c$-sortable elements
of~$W$ which maps~$w$ to the maximal $c$-sortable element below~$w$. Hence,~$w$ is
$c$-sortable if and only if~$\pi^c_\downarrow(w)=w$, \cite[Proposition~3.2]{reading3}.
It is easy to recover~$\pi^c_\downarrow$ in Figure~\ref{fig:cambrian_lattice_projection_123}
and~\ref{fig:cambrian_lattice_projection_213}. A $c$-sortable element~$w$ (white background)
is projected to itself; an element~$w$ which is not $c$-sortable (coloured background)
is projected to the (maximal) boxed $c$-sortable element below the coloured component
containing~$w$. For instance in Figure~\ref{fig:cambrian_lattice_projection_213},
we consider~$c=s_2s_1s_3$ and have $\pi^c_\downarrow(s_2s_3s_2)=s_2s_3s_2$ and
$\pi^c_\downarrow(s_3s_2s_1)=\pi^c_\downarrow(s_3s_2)=\pi^c_\downarrow(s_3)=s_3$.
\begin{figure}
      \begin{center}
      \begin{minipage}{0.95\linewidth}
         \begin{center}
         \begin{overpic}
           [width=0.9\linewidth]{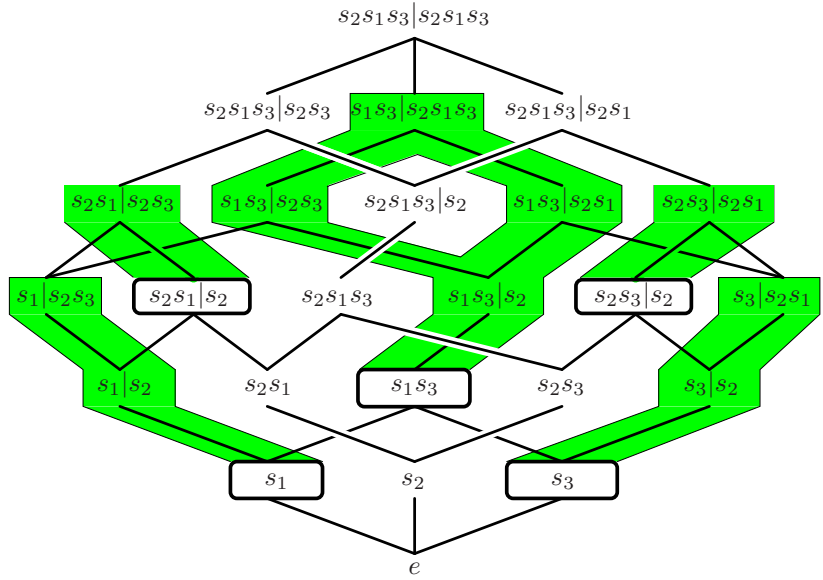}
            \put(40.5,70){$s_2s_1s_3|s_2s_1s_3$}
            \put(24,58.5){$s_2s_1s_3|s_2s_3$}
            \put(42,58.5){$s_1s_3|s_2s_1s_3$}
            \put(61,58.5){$s_2s_1s_3|s_2s_1$}
            \put(7.7,47.2){$s_2s_1|s_2s_3$}
            \put(26,47.2){$s_1s_3|s_2s_3$}
            \put(43.7,47.2){$s_2s_1s_3|s_2$}
            \put(62,47.2){$s_1s_3|s_2s_1$}
            \put(80.3,47.2){$s_2s_3|s_2s_1$}
            \put(1,35.5){$s_1|s_2s_3$}
            \put(17.3,35.5){$s_2s_1|s_2$}
            \put(36,35.5){$s_2s_1s_3$}
            \put(54,35.5){$s_1s_3|s_2$}
            \put(72,35.5){$s_2s_3|s_2$}
            \put(89,35.5){$s_3|s_2s_1$}
            \put(10.9,24.5){$s_1|s_2$}
            \put(29,24.5){$s_2s_1$}
            \put(47,24.5){$s_1s_3$}
            \put(65,24.5){$s_2s_3$}
            \put(83,24.5){$s_3|s_2$}
            \put(31.5,13){$s_1$}
            \put(48.3,13){$s_2$}
            \put(66.7,13){$s_3$}
            \put(49.2,2){$e$}
         \end{overpic}
         \end{center}
         \caption[]{$c=s_2s_1s_3$. }
         \label{fig:cambrian_lattice_projection_213}
      \end{minipage}
      \end{center}
\end{figure}

We say that~$w\in W$ is {\em $c$-antisortable} if~$ww_0$ is $c^{-1}$-sortable.
We have therefore a projection~$\pi_c^\uparrow$ from~$W$ to the set of $c$-antisortable
elements of~$W$ which takes~$w$ to the minimal $c$-antisortable element above~$w$.
For example in Figure~\ref{fig:cambrian_lattice_projection_213} we have
$\pi_c^\uparrow(s_1s_3)=s_1s_3s_2s_1s_3$.
The maps $\pi^c_\downarrow$ and $\pi_c^\uparrow$ have the same fibres, that is,
\[
  \left(\pi^c_\downarrow\right)^{-1}\pi^c_\downarrow(w)
    = \left(\pi_c^\uparrow\right)^{-1}\pi_c^\uparrow(w).
\]
These fibers are intervals in the weak order as shown by N.~Reading,~\cite[Theorem~1.1]{reading3}
and the fibre that contains~$w$ is $[\pi^c_\downarrow(w),\pi_c^\uparrow(w)]$.

\subsection{$c$-singletons}

We introduce a particularly important subclass of $c$-sortable elements: an element~$w\in W$
is a {\em $c$-singleton} if $\big(\pi^c_\downarrow\big)^{-1}(w)$ is a singleton. It is
easy to read off $c$-singletons in Figures~\ref{fig:cambrian_lattice_projection_123}
and~\ref{fig:cambrian_lattice_projection_213}: An element is a $c$-singleton if and only if
its background colour is white and it is not boxed, that is, $s_2s_1s_3$ in
Figure~\ref{fig:cambrian_lattice_projection_213} is a $c$-singleton while neither
$s_1s_3s_2$ nor $s_2s_3s_2$ are $c$-singletons.

\medskip
We now prove some useful properties of $c$-singletons.

\begin{prop}\label{prop:SingCFibers} Let~$w \in W$. The following propositions are equivalent.
\begin{compactenum}[(i)]
 \item $w$ is a $c$-singleton;
 \item $w$ is $c$-sortable and $ws$ is $c$-sortable for all $s\notin D(w)$;
 \item $w$ is $c$-sortable and $c$-antisortable.
\end{compactenum}
\end{prop}

\begin{proof}
   `(i) is equivalent to (iii)' and `(i) is equivalent to (ii)' follow from
    the fact that the fibre containing $w$ is $[\pi^c_\downarrow(w),\pi_c^\uparrow(w)]$
    and that the map~$\pi^c_\downarrow$ is order preserving.
\end{proof}
\noindent
It follows that~$w_0$ and~$e$ are $c$-singletons.

\medskip
The word property says that any pair of reduced expressions for~$w \in W$ can be linked
by a sequence of braid relation transformations. In particular, the set
\[
   S(w) := \{s_i\in S\,|\, w= s_1\dots s_{\ell(w)} \hbox{ is reduced}\}
         = \bigcap_{{\scriptstyle I\subset S \atop \scriptstyle w\in W_I}}I
\]
is independent of the chosen reduced expression for~$w$. It is clear that~$w\in W_{S(w)}$
and that $S(w)=K_1$ if~$w$ is $c$-sortable with $c$-factorization
$c_{(K_1)}c_{(K_2)}\dots c_{(K_p)}$.

Two reduced expressions for~$w\in W$ are equivalent {\em up to commutations} if they
are linked by a sequence of braid relation of order~$2$, that is, by commutations. Let~$\bf u,\bf w$
be reduced expressions for~$u,w \in W$. Then~$\bf u$ is a {\em prefix of~$\bf w$ up to
commutations} if~$\bf u$ is the prefix of a reduced expression~${\bf w'}$ and~${\bf w'}$ is equivalent
to~$\bf w$ up to commutations. We now state the main result of this section.
Its proof is deferred until
after Proposition~\ref{prop:Prefixes}.

\begin{thm}\label{thm:cSingleton}
   Let~$w$ be in~$W$. Then~$w$ is a $c$-singleton if and only if~$w$
   is a prefix of $\ww$ up to commutations.
\end{thm}

\begin{rem}
   \textnormal{For computational purposes, it would be interesting to find a nice combinatorial
               description of $\ww$.}
\end{rem}

\begin{expl}
   \textnormal{Let $W=S_4$ with set of generators $S= \{s_i\,|\,1\leq i\leq 3\}$
   and Coxeter element~$c=s_2s_1s_3$. The $c$-singletons of~$W$ are
   \begin{center}
   \begin{tabular}{lll}
      $e$,            & $s_2s_3$,             & $s_2s_1s_3s_2s_1$,\\
      $s_2$,          & $s_2s_1s_3$,          & $s_2s_1s_3s_2s_3$, and\\
      $s_2s_1,\qquad$ & $s_2s_1s_3s_2,\qquad$ & $w_0=s_2s_1s_3s_2s_1s_3$.
   \end{tabular}
   \end{center}
   We see here that~$s_2s_3$ is not a prefix of~$s_2s_1s_3s_2s_1s_3$, but it does appear as
   a prefix after commutation of the commuting simple reflections~$s_1$ and~$s_3$.}
\end{expl}

\begin{prop}
   The $c$-singletons constitute a distributive sublattice of the (right) weak order on~$W$.
\end{prop}
\noindent
Examples of these distributive lattices for~$W=S_4$ are given in Figure~\ref{fig:dist_lattices}.

\begin{figure}[b]
  \begin{center}
  \begin{minipage}{0.95\linewidth}
     \begin{center}
     \begin{overpic}
        [width=10cm]{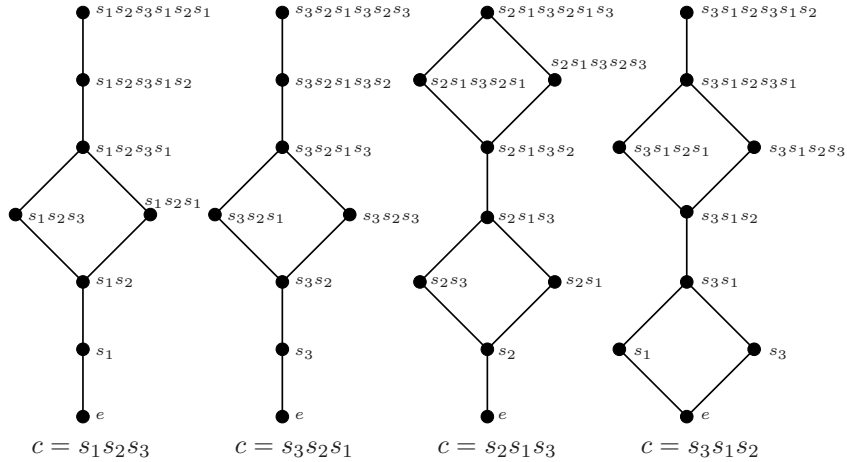}
        \put(11.5,0.5){\tiny $e$}
        \put(11.5,9){\tiny $s_1$}
        \put(11.5,18.5){\tiny $s_1s_2$}
        \put(2.5,27){\tiny $s_1s_2s_3$}
        \put(18,29){\tiny $s_1s_2s_1$}
        \put(11.5,36){\tiny $s_1s_2s_3s_1$}
        \put(11.5,45){\tiny $s_1s_2s_3s_1s_2$}
        \put(11.5,54.5){\tiny $s_1s_2s_3s_1s_2s_1$}
        \put(3,-4){$c=s_1s_2s_3$}
        \put(38,0.5){\tiny $e$}
        \put(38,9){\tiny $s_3$}
        \put(38,18.5){\tiny $s_3s_2$}
        \put(29,27){\tiny $s_3s_2s_1$}
        \put(47,27){\tiny $s_3s_2s_3$}
        \put(38,36){\tiny $s_3s_2s_1s_3$}
        \put(38,45){\tiny $s_3s_2s_1s_3s_2$}
        \put(38,54.5){\tiny $s_3s_2s_1s_3s_2s_3$}
        \put(30,-4){$c=s_3s_2s_1$}
        \put(65,0.5){\tiny $e$}
        \put(65,9){\tiny $s_2$}
        \put(56,18.5){\tiny $s_2s_3$}
        \put(74,18.5){\tiny $s_2s_1$}
        \put(65,27){\tiny $s_2s_1s_3$}
        \put(65,36){\tiny $s_2s_1s_3s_2$}
        \put(56,45){\tiny $s_2s_1s_3s_2s_1$}
        \put(72,47.5){\tiny $s_2s_1s_3s_2s_3$}
        \put(65,54.5){\tiny $s_2s_1s_3s_2s_1s_3$}
        \put(57,-4){$c=s_2s_1s_3$}
        \put(92,0.5){\tiny $e$}
        \put(83,9){\tiny $s_1$}
        \put(101,9){\tiny $s_3$}
        \put(92,18.5){\tiny $s_3s_1$}
        \put(92,27){\tiny $s_3s_1s_2$}
        \put(83,36){\tiny $s_3s_1s_2s_1$}
        \put(101,36){\tiny $s_3s_1s_2s_3$}
        \put(92,45){\tiny $s_3s_1s_2s_3s_1$}
        \put(92,54.5){\tiny $s_3s_1s_2s_3s_1s_2$}
        \put(84,-4){$c=s_3s_1s_2$}
     \end{overpic}
     \end{center}$ $
     \caption[]{There are four Coxeter elements in~$S_4$. Each yields a distributive lattice of $c$-singletons.}
     \label{fig:dist_lattices}
  \end{minipage}
  \end{center}
\end{figure}

\begin{proof}
   Let~$L$ be the set of subsets $P\subset \{1,\dots,\ell(w_0)\}$ with the property
   that the reflections at positions $i \in P$ of~$\ww$ can be moved by commutations
   to form a prefix~${\bf w}_P$ of~$\ww$. This prefix~${\bf w}_P$ represents~$w_P \in W$.
   Note that $\ell(w_P)=|P|$ because ${\bf w}_P$ is a prefix up to commutation of a
   reduced word for~$\ww$. The set~$L$ is partially ordered by  inclusion and forms
   a distributive lattice with $P_1\vee P_2 = P_1\cup P_2$ and $P_1\wedge P_2 = P_1\cap P_2$
   according to~\cite[Exercise 3.48]{stanley}. (In particular, $P_1\cup P_2,\ P_1\cap P_2\in L$
   if $P_1, P_2\in L$).

   We claim that $P \longmapsto w_P$ is an injective lattice homomorphism.

   First we check injectivity. Suppose $w_P=w_Q$ for $P\ne Q$. Since~${\bf w}_P$
   and~${\bf w}_Q$ are reduced expressions, we have $|P|=|Q|=r$. Let $P=\{p_1,\dots,p_r\}$
   and $Q=\{q_1,\dots,q_r\}$ with $p_i<p_{i+1}$ and $q_i < q_{i+1}$. Without loss 
   of generality, let the smallest element in $(P\cup Q)\setminus (P\cap Q)$ be~$p_i$.
   Let~$s\in S$ be the reflection at position~$p_i$ of~$\ww$, then~$s$ also appears 
   at some~$q_j$ with~$q_j>p_i$. When moving the reflections in $Q$ to the front 
   of~$\ww$, the $s$ that started at $q_j$ must pass the $s$ at $p_i$, but this implies 
   that the expression for $\ww$ would not be reduced at this step, which is contrary 
   to our assumption.  Thus the map is injective.

   We show that $P \longmapsto w_P$ respects the lattice structures of~$L$ and~$W$.
   Let $P,Q\in L$ and $R=P\cap Q$. Since~$R\in L$,~${\bf w}_R$ is a prefix of~$\ww$
   up to commutation. In particular, it is also a prefix of~${\bf w}_P$ and~${\bf w}_Q$
   up to commutation. Hence $w_R\leq w_P$ and $w_R \leq w_Q$. We obtain~${\bf w}_{P\setminus R}$
   and~${\bf w}_{Q\setminus R}$ from~${\bf w}_P$ and~${\bf w}_Q$ by deletion of all
   reflections that correspond to an element of~$R$ and
   conclude~$w_P={w}_{R}{w}_{P\setminus R}$ and~$w_Q={w}_{R}{w}_{Q\setminus R}$.
   We have $S({\bf w}_{P\setminus R}) \cap S({\bf w}_{Q\setminus R}) = \varnothing$
   since~$\ww$ is reduced. The proof is by contradiction and is similar to the proof of
   injectivity. Therefore, no element of~$W$ is above~${\bf w}_R$ and below~${\bf w}_P$
   and~${\bf w}_Q$. We have shown $w_R = w_P \wedge w_Q$ with respect to the weak order
   of~$W$. A similar argument proves~$w_T=w_P\vee w_Q$ with respect to the weak order
   of~$W$ where~$T = P \cup Q$:
   $S({\bf w}_{T\setminus P}) \cap S({\bf w}_{T\setminus Q}) = \varnothing$
   implies that no~$w\in W$ below~$w_R$ and above~$w_P$ and~$w_Q$ exists.
\end{proof}

The following lemma characterizes the elements that cover a $c$-singleton.

\begin{lem}\label{lem:CSing}
   Let~$c_{(K_1)}\ldots c_{(K_p)}$ be the $c$-factorization of the $c$-singleton~$w$
   and $s\notin D(w)$. The $c$-factorization of the cover~$ws$ of $w$ is either
   $c_{(K_1)}\ldots c_{(K_p)}c_{(s)}$ or
   $c_{(K_1)}\ldots c_{(K_i\cup\{s\})}\ldots c_{(K_p)}$.

   If $ws=c_{(K_1)}\ldots c_{(K_i\cup\{s\})}\ldots c_{(K_p)}$ then~$i$ is uniquely
   determined and~$s$ commutes with every~$r\in K_{i+1}\cup L$ where~$L$ satisfies
   $c_{(K_i\cup\{s\})}=c_{(K_i\setminus L)}\,s\,c_{(L)}$.
\end{lem}
\begin{proof}
   If $s\in K_p$ then $c_{(K_1)}\ldots c_{(K_p)}c_{(s)}$ is obviously the
   $c$-factorization for~$ws$. So we assume $s\notin K_p$. As~$w$ is a
   $c$-singleton, $ws$ is $c$-sortable with $c$-factorization $c_{(L_1)}\ldots c_{(L_q)}$
   where $L_1 \supseteq \ldots \supseteq L_q$. As $s\in D(ws)$, there is a
   unique $1\leq i\leq q$ and~$r\in L_i$
   such that $w=(ws)s=c_{(L_1)}\dots c_{(L_i\setminus\{r\})}\dots c_{(L_q)}$
   by the exchange condition.

   \medskip
   \noindent
   \textbf{Case 1:} Suppose $i=1$, i.e. $i=1$ is the unique index such that
   \begin{equation}\label{equ:Proof}
      w=(ws)s=c_{(L_1\setminus\{r\})} c_{(L_2)}\ldots c_{(L_q)}.
   \end{equation}
   \textbf{Case 1.1:} $r\notin K_1$.
      Then $r\notin S(w)$ and $s=r$ because $r\in S(ws)=S(w) \cup \{s\}$. Since any
      two reduced expressions of~$ws$ are linked by braid relations according to Tits'
      Theorem,~\cite[Theorem 3.3.1]{bjorner-brenti} and since $s\notin S(w)$, we conclude
      that we have to move~$s$ from the rightmost position to the left by commutation.
      In other words,~$s$ commutes with $K_2\cup L$.\\
   \textbf{Case 1.2:} $r\in K_1=S(w)$.
      As $c_{(L_1\setminus\{r\})} c_{(L_2)}\ldots c_{(L_q)}$ is reduced and
      $L_2 \supseteq \dots \supseteq L_q$ is nested, we have $r\in L_2$. Hence
      \[ K_1=S(w)=(L_1\setminus\{r\})\cup L_2=L_1\cup L_2=L_1.\]
      Thus $c_{(L_2)}\dots  c_{(L_q)}$ and $c_{(K_2)}\dots  c_{(K_p)} s$ are reduced
      expressions for some $\widehat w \in W$ and $s\in D(\widehat w)$. The exchange
      condition implies the existence of a unique index $2\leq j\leq q$
      and $t\in L_j$ such that
      \[ \widehat ws=c_{(L_2)}\ldots c_{(L_j\setminus\{t\})}\ldots c_{(L_q)}.\]
      In other words
      \[ w=c_{(L_1)}\hat w= c_{(L_1)}c_{(L_2)}\ldots c_{(L_j\setminus\{t\})}\ldots c_{(L_q)}\]
      is reduced. But this contradicts the uniqueness of
      $i=1$ in Equation~(\ref{equ:Proof}). So $r \notin K_{1}$ and we finished the first case.

   \medskip
   \noindent
   \textbf{Case 2:} Suppose $i>1$, then  $K_1=S(w)=L_1$.
      Set $\nu:=\min(p,i-1)$ and iterate the argument for $c_{(L_1)}^{-1}w$,
      $c_{(L_2)}^{-1}c_{(L_1)}^{-1}w,\ldots$ to conclude $L_j=K_j$
      for $1\leq j \leq \nu$. If $\nu=p$ then $i=q=p+1$ and $L_i\setminus\{r\}=\varnothing$.
      So $L_{i}=\{s\}\subseteq L_{i-1}=K_p$ which contradicts $s\notin K_p$. Thus $\nu =i-1$
      for some $i\leq p$ and $L_j=K_j$ for $1 \leq j \leq i-1$. We may assume $i=1$ and are
      done by Case 1.
\end{proof}

\begin{prop}\label{prop:Prefixes}
   Let~$w$ be a $c$-singleton and~$\bf w$ its $c$-sorting word. Any prefix of~$\bf w$
   up to commutations is a $c$-singleton.
\end{prop}
\begin{proof}
   Let $c_{(K_1)}\ldots c_{(K_p)}$ denote the $c$-factorization of~$w$. It is sufficient
   to show that the prefix $w^{\prime}$ up to commutations of length~$\ell(w)-1$ is a
   $c$-singleton. There is $1 \leq i \leq p$ and $r\in K_p$ such that
   $w^{\prime}=c_{(K_1)}\dots c_{(K_i\setminus\{r\})}\dots c_{(K_p)}$ is the
   $c$-factorization of $w^{\prime}$. It remains to show that $w^{\prime}s$ is $c$-sortable
   for $s \notin D(w^{\prime})$.

   \medskip
   \noindent
   \textbf{Case 1:} Suppose $s\in D(w)$. Recall the definition of the
   {\em Bruhat order $\leq_\mathcal{B}$ on~$W$}: $u\leq_\mathcal{B} v$ in~$W$
   if an expression for~$u$ can be obtained as a subword of a reduced expression of~$v$,
   see~\cite[Chapter~2]{bjorner-brenti}. The lifting property of the Bruhat order
   implies $w^{\prime}s\leq_{\mathcal B} w$. Moreover
   $\ell(w^{\prime}s)=\ell(w^{\prime})+1=\ell(w)$. Thus $w=w^{\prime}s$
   and $s=r$. In particular $w^{\prime}s=w$ is $c$-sortable.

   \medskip
   \noindent
   \textbf{Case 2:} Suppose $s\notin D(w)$, in particular $s\not=r$. So
   $ws$ is $c$-sortable and by Lemma~\ref{lem:CSing} there are two cases to
   distinguish. \\
   \textbf{Case 2.1:} If $c_{(K_1)}\ldots c_{(K_p)}c_{(s)}$ is the $c$-factorization
   of~$ws$ then $s\in K_p$ and $c_{(K_1)}\dots c_{(K_i\setminus\{r\})}\dots c_{(K_p)}c_{(s)}$
   is the $c$-factorization of~$w^{\prime}s$. In particular, the sequence
   $K_1\supseteq \ldots \supseteq K_{i}\setminus\{r\} \supseteq \ldots \supseteq K_p \supseteq \{s\}$
   is nested and~$w^{\prime}s$ is $c$-sortable.\\
   \textbf{Case 2.2:} If $c_{(K_1)}\ldots c_{(K_j\cup\{s\})}\ldots c_{(K_p)}$
   is the $c$-factorization of $ws$ then either $s$ and $r$ commute or not.

   If~$s$ and~$r$ do not commute then $j = i$ and~$r$ appears before~$s$ in the
   chosen reduced expression of~$c$, since $s$ commutes with every simple reflections
   to the right of the rightmost copy of~$s$ in the $c$-factorization of~$ws$ by
   Lemma~\ref{lem:CSing}. Then
   $w^{\prime}s=c_{(K_1)}\ldots c_{(K_i\setminus\{r\}\cup\{s\})}\ldots c_{(K_p)}$
   is $c$-sortable.

   If~$s$ and~$r$ commute, suppose first~$j \leq i$. Then
   \[
     w's = wrs = wsr
         = c_{(K_1)}\ldots c_{(K_j\cup\{s\})}\ldots c_{(K_i\setminus\{r\})}\ldots c_{(K_p)}
   \]
   is the $c$-factorization of~$w's$. As
   $K_1 \supseteq \ldots
        \supseteq K_j\cup\{s\}
        \supseteq \ldots
        \supseteq K_i\setminus\{r\}
        \supseteq \ldots
        \supseteq K_p$
   is nested,~$w^{\prime}s$ is $c$-sortable. The case $j > i$ is proved similarly.

   We conclude that~$w^{\prime}s$ is $c$-sortable for any $s\notin D(w')$,
   so~$w'$ is a $c$-singleton.
\end{proof}

\begin{proof}[Proof of Theorem~\ref{thm:cSingleton}]
   We know by Proposition~\ref{prop:SingCFibers} that~$w$ is a $c$-singleton if
   and only if~$w$ is $c$-sortable and~$ww_0$ is $c^{-1}$-sortable.

   Suppose~$w$ is a $c$-singleton. Let $s$ be the rightmost simple reflection appearing
   in the $c^{-1}$-factorization for $ww_0$, so $ww_0=us$ for some $c^{-1}$-sortable
   element~$u$.

   We have~$su^{-1}w=w_0$ and hence~$u^{-1}w=sw_0$. Since~$S=w_0 S w_0$,~$t:=w_0sw_0$
   is a simple reflection. Now $u^{-1}wt=w_0$ implies $\ell(wt)>\ell(w)$ and
   we conclude that~$wt$ is $c$-sortable by Proposition~\ref{prop:SingCFibers}.
   But~$wt$ is also $c$-antisortable since $wtw_0=u$ is $c^{-1}$-sortable. Hence,~$wt$
   is a $c$-singleton that covers~$w$ in the weak order.

   Repeating this process, we show that every $c$-singleton is on an unrefinable chain
   of $c$-singletons leading up to~$w_0$. By downwards induction, every element of that
   chain is a prefix of~$w_0$ up to commutations. This is clearly true for~$w_0$. As we
   went up each step, though, we added a simple reflection which commuted with every
   reflection to its right (or was added at the rightmost end), by Lemma~\ref{lem:CSing}.
   Thus, when we want to remove the element we added at the last step, we can rewrite~$w_0$
   using commutations only such that this simple reflection is on the right.
\end{proof}

\section{Coxeter fans, Permutahedra, and Cambrian fans}\label{se:CambFan}

In this section, we describe the geometry of Coxeter fans and $c$-Cambrian fans.
We first recall some fact about the geometric representation of~$W$ and use the
notation of~\cite{humphreys} on Coxeter groups and root systems. Let~$W$ act by
reflections on an $\mathbb R$-Euclidean space $(V,\scalprod{\cdot}{\cdot})$.

Let~$\Phi$ be a root system corresponding to~$W$ with simple roots~$\Delta=\{\a_s\,|\, s\in S\}$,
positive roots~$\Phi^+=\Phi\cap\mathbb R_{>0}[\Delta]$ and negative roots~$\Phi^-=-\Phi^+$.
Without loss of generality, we assume that the action of~$W$ is essential relative
to~$V$, that is,~$\Delta$ is a basis of~$V$. The set~$\Phi^+$ parametrizes the set
of reflections in~$W$: to each reflection~$t\in W$ there corresponds a unique positive
root~$\a_t\in\Phi^+$ such that~$t$ maps~$\a_t$ to~$-\a_t$ and fixes the
hyperplane~$H_t=\{v\in V\,|\,\scalprod{v}{\a_t}=0\}$.

The {\em Coxeter arrangement}~$\A$ for~$W$ is the collection of all reflecting
hyperplanes for~$W$. The complement~$V\setminus(\bigcup\A)$ of~$\A$ consists of
open cones. Their closures are called {\em chambers}. The chambers are in canonical
bijective correspondence with the elements of~$W$. The {\em fundamental chamber}
$D:=\bigcap_{s \in S} \{ v\in V\,|\,\langle v,\alpha_s \rangle \geq 0 \}$ corresponds
to the identity~$e \in W$ and the chamber~$w(D)$ corresponds to~$w\in W$.

A subset~$U$ of~$V$ is {\em below} a hyperplane $H\in\A$ if every point in~$U$ is
on~$H$ or on the same side of~$H$ as~$D$. The subset~$U$ is {\em strictly below}~$H\in\A$
if~$U$ is below~$H$ and~$U \cap H = \varnothing$. Similarly,~$U$ is {\em above} or
{\em strictly above} a hyperplane~$H\in\A$. The inversions of~$w\in W$ are the
reflections that correspond to the hyperplanes~$H$ which $w(D)$ is above.

For a simple reflection~$s\in S$, we have $\ell(sw)<\ell(w)$ if and only if~$s\le w$ in the
weak order if and only if~$w(D)$ is above~$H_s$. To decide whether~$w(D)$ is above or
below~$H_s$ is therefore a weak order comparison. These notions will be handy
in \S\ref{se:CamPoly}.

\medskip
\noindent A {\em fan}~$\G$ is a family of nonempty closed polyhedral
(convex) cones in~$V$ such that
\begin{compactenum}[(i)]
   \item every face of a cone in~$\G$ is in~$\G,$ and
   \item the intersection of any two cones in $\G$ is a face of both.
\end{compactenum}
A fan~$\G$ is {\em complete} if the union of all its cones is~$V$,
{\em essential} (or {\em pointed}) if the intersection of all
non-empty cones of~$\G$ is the origin, and {\em simplicial} if every
cone is simplicial, that is, spanned by linearly independent
vectors. A $1$-dimensional cone is called a {\em ray} and a ray
is~{\em extremal} if it is a face of some cone. The set of
$k$-dimensional cones of~$\G$ is denoted by $\G^{(k)}$ and two cones
in~$\G^{(k)}$ are {\em adjacent} if they have a common face in
$\G^{(k-1)}$. A fan~$\G$ coarsens a fan~$\G^{\prime}$ if every cone
of~$\G$ is the union of cones of~$\G^{\prime}$ and
$\bigcup_{C\in\G}C = \bigcup_{C\in\G^\prime}C$. We refer
to~\cite[Lecture~7]{Ziegler} for more details and examples.

\medskip 
The chambers and all their faces of a Coxeter arrangement~$\A$ define the
{\em Coxeter fan}~$\F$. The Coxeter fan~$\F$ is known to be complete,
essential, and simplicial,~\cite[Sections~1.12--1.15]{humphreys}. The
fundamental chamber~$D\in \F$ is a (maximal) cone spanned by the (extremal)
rays $\{\rho_s\,|\, s\in S\}$, where~$\rho_s$ is the intersection of~$D$
with the subspace orthogonal to the hyperplane spanned by $\{\a_t\,|\,t\in\la s\ra\}$.

Recall that the rays of~$\F$ decompose into~$n$ orbits under the action of~$W$ and
each orbit contains exactly one~$\rho_s$,~$s\in S$. Thus, any ray~$\rho\in\F^{(1)}$
is~$w(\rho_s)$ for some~$w\in W$ where~$s \in S$ is uniquely determined by~$\rho$
but~$w$ is not unique. In fact, $w(\rho_s)=g(\rho_s)$ if and only if $w\in gW_{\la s\ra}$.

\subsection{Permutahedra}\label{subsection:permutahedra}

We illustrate Coxeter fans by means of permutahedra, that is, polytopes
that have a Coxeter fan as normal fan.

Take a point~$\pmb a$ of the complement~$V\setminus(\bigcup\A)$ of the Coxeter
arrangement~$\A$, and consider its $W$-orbit. The convex hull of this $W$-orbit
is a {\em $W$-permutahedron} denoted by~$\Perm^{\pmb a}(W)$. There is a bijection
between the rays of~$\F$ and the facets of~$\Perm^{\pmb a}(W)$: there is a
halfspace associated to each ray~$\rho \in \F$ such that its supporting hyperplane
is perpendicular to~$\rho$ and such that the permutahedron is the intersection of
these halfspaces. Let us be more precise.

Let $\Delta^*:=\{v_s \in V\,|\, s\in S\}$ be the fundamental weights
of~$\Delta$, that is, $\Delta^*$ is the dual basis of~$\Delta$ for
the scalar inner product. The fundamental chamber~$D$ is spanned by
the fundamental weights, that is, $D=\mathbb R_{\geq 0}[\Delta^*]$.
Hence, the rays of~$\A$ are easily expressed in terms of~$\Delta^*$:
We have $\rho_s=\mathbb R_{\geq 0}[v_s]$ and therefore
$w(\rho_s)=\mathbb R_{\geq 0}[w(v_s)]$ for any $w\in W$ and $s\in
S$.

Without loss of generality, we choose~$\pmb a=\sum_{s\in S} a_s v_{s}$ in the
interior of~$D$, that is $a_{s} > 0$ for ~$s\in S$, and define $M(w):=w(\pmb a)$.
All points~$M(w)$ are distinct and the convex hull of $\{M(w)\,|\, w\in W\}$ yields
a realization of the $W$-permutahedron~$\Perm^{\pmb a}(W)$. It is not difficult to
describe this polytope as an intersection of half-spaces.

For each~$\rho=w(\rho_s)\in \F^{(1)}$, we define the closed half space
\[
   \H_{\rho}^{\pmb a}
    := \{ v \in V\,|\,
           \scalprod{v}{w(v_s)}\leq \scalprod{\pmb a}{v_{s}}\}.
\]
This definition does not depend on the choice of~$w\in W$ such that~$\rho=w(\rho_s)$,
but only of the coset $W/W_{\la s\ra}$. The open half space~$\H^{{\pmb a},+}_{\rho}$
and the hyperplane~$H_{\rho}^{\pmb a}$ are defined by strict inequality and equality
respectively. Now, the permutahedron~$\Perm^{\pmb a}(W)$ is given by
\[
  \Perm^{\pmb a}(W) = \bigcap_{\rho\in\F^{(1)}}
                      \H_{\rho}^{\pmb a}.
\]
As for the rays of the Coxeter fan, we have 
$H_{\rho}^{\pmb a}=H_{(w,s)}^{\pmb a}=H_{(w^\prime,s)}^{\pmb a}$ 
if and only if $w\in w^\prime W_{\la s\ra}$ and 
$\rho=w(\rho_s)=w^\prime(\rho_s)$. Moreover, 
$M(w)\in H_{(w^\prime,s)}^{\pmb a}$ if and only if 
$H_{(w^\prime,s)}^{\pmb a}=H_{(w,s)}^{\pmb a}$. A 
simple description of the vertex~$M(w)$ of the 
permutahedron follows:
\[
   M(w) = \bigcap_{s\in S} H_{(w,s)}^{\pmb a} .
\]

\begin{expl}[Realization of~$\Perm(S_3)$]
   \textnormal{We consider the Coxeter group~$W=S_3$ of type~$A_2$ acting
       on~$\mathbb R^2$. The reflections~$s_1$ and~$s_2$ generate~$W$ and
       the simple roots that correspond to~$s_1$ and~$s_2$ are~$\a_1$ and~$\a_2$.
       They are normal to the reflection hyperplanes~$H_{s_1}$ and~$H_{s_2}$.
       The fundamental weight vectors that correspond to the simple roots are the vectors~$v_1$
       and~$v_2$ and determine the ray~$L = \{ \mu (a_1 v_1 + a_2 v_2) \;|\; \mu >0\}$,
       $a_1,a_2 > 0$, which contains~$M(e)=\pmb a \in L$. We obtain the
       permutahedron~$\Perm(S_3)$ as convex hull of the $W$-orbit of~$M(e)$.
       Alternatively, the permutahedron is described as intersection of the
       half spaces~$\H_{(x,s)}^{\pmb a}$ with bounding hyperplanes~$H_{(x,s)}^{\pmb a}$ for~$x\in W$
       and~$s \in S$. All objects are indicated in Figure~\ref{fig:construction_perm}.}
\end{expl}

\begin{figure}
  \begin{center}
  \begin{minipage}{0.95\linewidth}
     \begin{center}
     \begin{overpic}
        [width=140pt]{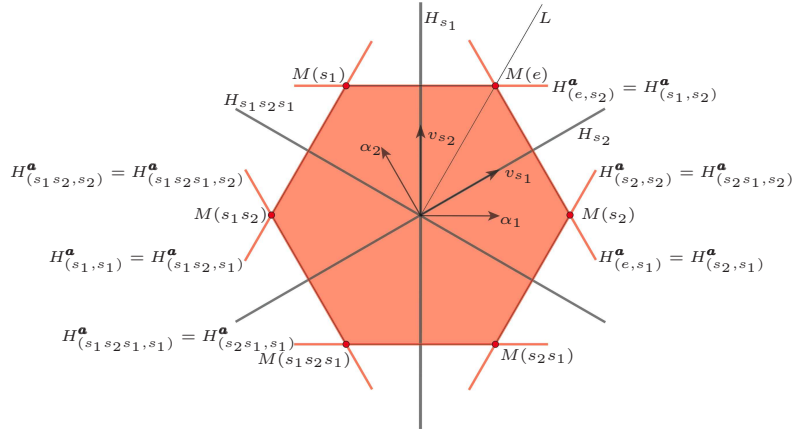}
        \put(71,95){\tiny $L$}
        \put(63,82){\tiny $M(e)$}
        \put(44,95){{\tiny {$H_{s_1}$}}}
        \put(13,82){\tiny $M(s_1)$}
        \put(-10,49){\tiny $M(s_1s_2)$}
        \put(81,49){\tiny $M(s_2)$}
        \put(62,16){\tiny $M(s_2s_1)$}
        \put(5,15){\tiny $M(s_1s_2s_1)$}
        \put(80,68){\tiny $H_{s_2}$}
        \put(-3,76){\tiny $H_{s_1s_2s_1}$}
        \put(29,65){\tiny $\a_2$}
        \put(62,48){\tiny $\a_1$}
        \put(63,59){\tiny $v_{s_1}$}
        \put(45,68){\tiny $v_{s_2}$}
        \put(85,39){\tiny $H_{(e,s_1)}^{\pmb a}=H_{(s_2,s_1)}^{\pmb a}$}
        \put(85,59){\tiny $H_{(s_2,s_2)}^{\pmb a}=H_{(s_2s_1,s_2)}^{\pmb a}$}
        \put(74,79){\tiny $H_{(e,s_2)}^{\pmb a}=H_{(s_1,s_2)}^{\pmb a}$}
        \put(-44,39){\tiny $H_{(s_1,s_1)}^{\pmb a}=H_{(s_1s_2,s_1)}^{\pmb a}$}
        \put(-53,59){\tiny $H_{(s_1s_2,s_2)}^{\pmb a}=H_{(s_1s_2s_1,s_2)}^{\pmb a}$}
        \put(-41,20.5){\tiny $H_{(s_1s_2s_1,s_1)}^{\pmb a}=H_{(s_2s_1,s_1)}^{\pmb a}$}
     \end{overpic}
     \end{center}
     \caption[]{The permutahedron~$\Perm(S_3)$ obtained as convex hull of the $S_3$-orbit
                of~$M(e) \in L$ or as intersection of the half spaces~$\H_{(x,s)}^{\pmb a}$.}
     \label{fig:construction_perm}
  \end{minipage}
  \end{center}
\end{figure}

\subsection{Cambrian fans}
For any lattice congruence~$\Theta$ of the weak order on~$W$, N.~Reading constructs
a complete fan~$\F_\Theta$ that coarsens the Coxeter fan~$\F$,~\cite{reading0}. A
maximal cone~$C_\vartheta\in \F_\Theta$ corresponds to a congruence class~$\vartheta$
of~$\Theta$ and~$C_\vartheta$ is the union of the chambers of~$\A$ that correspond
to the elements of~$\vartheta$. In~\cite[Section~5]{reading0} N.~Reading proves that these
unions are indeed convex cones and that the collection~$\F_\Theta$ of these
cones and their faces is a complete fan.

The {\em $c$-Cambrian fan}~$\F_c$ of~$W$ is obtained by this construction if we
consider the lattice congruence with congruence classes~$[\pi^c_\downarrow(w),\pi^\uparrow_c(w)]$
for~$w\in W$ and chosen Coxeter element~$c$. The $n$-dimensional cone that corresponds to the
$c$-sortable element~$w$ is denoted by~$C(w)$. It is the union of the maximal
cones of~$\F$ that correspond to the elements
of~$(\pi^c_\downarrow)^{-1}\pi^c_\downarrow(w)=[\pi^c_\downarrow(w),\pi^\uparrow_c(w)]$.
In particular,~$C(w)$ is a maximal cone of~$\F_c$ and of~$\F$ if and only if~$w$
is a $c$-singleton.

In~\cite{reading4}, N.~Reading and D.~Speyer define a bijection between the set of rays
of~$\F_c$ and the set of {\em almost positive roots}~$\Phi_\apr:=\Phi^+\cup(-\Delta)$.
To describe this labeling of the rays, we first define a set of almost positive roots for
any $c$-sortable~$w$. For $s\in S(w)$, let $1\leq j_s\leq \ell(w)$ be the unique integer
such that~$s_{j_s}$ is the rightmost occurrence of~$s$ in the $c$-sorting
word~$s_{1}\ldots s_{\ell(w)}$ of~$w$ and define
\[
  \lr_s(w):= \begin{cases}
                s_1 \ldots s_{j_s-1}(\a_s) & \text{ if }s\in S(w)\\
                -\alpha_s                  &\textrm{ if }s\not\in S(w),
             \end{cases}
  \quad\text{and}\quad
  \cl_c(w) := \bigcup_{s\in S} \lr_s(w).
\]

\begin{expl}
   \textnormal{To illustrate these maps, we consider the Coxeter group~$W=S_3$ with
        generators $S=\{s_1,s_2\}$ as shown in Figure~\ref{fig:construction_perm}.
        Choose~$c=s_1s_2$ as Coxeter element. It is easy to check
        that~$w\in W\setminus\{s_2s_1\}$ is $c$-sortable and that
        $w \in W\setminus\{s_2, s_2s_1\}$ is a $c$-singleton. From the above definition
        follows
        \begin{center}
        \begin{tabular}{lll}
           $\lr_{s_1}(e)=\lr_{s_1}(s_2)=-\a_1$,      & & $\lr_{s_2}(e)=\lr_{s_2}(s_1)=-\a_2$,\\
           $\lr_{s_1}(s_1)=\lr_{s_1}(s_1s_2)=\a_1$,  & & $\lr_{s_2}(s_1s_2s_1)
                                                               =\lr_{s_2}(s_1s_2)=\a_1+\a_2$,\\
           $\lr_{s_1}(s_1s_2s_1)=\a_2$,              & & $\lr_{s_2}(s_2)=\a_2$,
        \end{tabular}
        \end{center}
        and therefore
        \begin{center}
        \begin{tabular}{lll}
           $\cl_{c}(e)  =\{ -\a_1,-\a_2\}$,  & & $\cl_{c}(s_1s_2)=\{ \a_1, \a_1+\a_2\}$,\\
           $\cl_{c}(s_1)=\{ \a_1, -\a_2\}$,  & & $\cl_{c}(s_1s_2s_1) =\{ \a_2, \a_1+\a_2\}$,\\
           $\cl_{c}(s_2)=\{ -\a_1, \a_2\}$.  & &
        \end{tabular}
        \end{center}
    }
\end{expl}

N.~Reading and D.~Speyer use the cluster map~$\cl_c$ to prove that $c$-Cambrian fans and
cluster fans have the same combinatorics: the maximal cone~$C(w)$ of the $c$-Cambrian fan
represented by the $c$-sortable element~$w$ is mapped to the set~$\cl_c(w)$ of almost
positive roots. The cardinality of~$\cl_c(w)$ matches the number of extremal rays of~$C(w)$
and~$\cl_c$ induces a bijection~$f_c$ between the rays of~$\F_c$ and the almost positive
roots by extending~$\cl_c$ to intersection of cones:
$\cl_c(C_1\cap C_2):=\cl_c(C_1)\cap\cl_c(C_2)$.
To put it slightly differently, N.~Reading and D.~Speyer showed the following Theorem.

\begin{thm}[Reading-Speyer~{\cite[Theorem 7.1]{reading4}}]\label{thm:ReadSpeyer}
   There is a bijective labeling~$f_c:\F_c^{(1)}\rightarrow \Phi_{\geq -1}$ of the
   rays of the $c$-Cambrian fan~$\F_c$ by almost positive roots such that the extremal
   rays of~$C(w)$ are labeled by~$\cl_c(w)$.
\end{thm}

We now aim for an explicit description of~$f_c$ that relates nicely to $c$-singletons,
but first need the following two lemmas.

\begin{lem}\label{lem:LemA}
   For~$\beta\in\Phi_{\apr}$, there exists a $c$-sortable element~$w$ and a simple
   reflection~$s$ such that~$\lr_s(w)=\beta$.
\end{lem}

\begin{proof}
   The identity~$e$ is a $c$-singleton and~$\cl_c(e) = \Phi_{\geq -1}\setminus \Phi^+$,
   so we are done if $\beta$ is a negative simple root. Suppose that~$\beta\in \Phi^+$
   and consider the longest element~$w_0$ with $c$-sorting word
   ${\ww}=s_{j_1}s_{j_2}\dots s_{j_N}$. Since~$w_0(D)$ is above all reflecting hyperplanes,
   \[
     \Phi^+=\{s_{j_1}s_{j_2}\dots s_{j_{p-1}}(\a_{s_{j_p}})\,|\,1\leq p\leq N\}
   \]
   and~$\beta = s_{j_1}\dots s_{j_{i-1}}(\alpha_{s_{j_i}})$ for some~$1\leq i\leq N$.
   Since~$w=s_{j_1}\dots s_{j_i}$ is a prefix of~$\ww$, it is a
   $c$-singleton and $\lr_{s_{j_i}}(w)=\beta$.
\end{proof}

\begin{lem}\label{lem:enoughsing}
   Let $\rho\in\F_c^{(1)}$. There is a $c$-singleton~$w$ such that~$\rho$ is an
   extremal ray of $C(w)$.
\end{lem}

\begin{proof}
   Pick~$\rho\in\F_c^{(1)}$. According to Theorem~\ref{thm:ReadSpeyer},~$f_c(\rho)=\beta$
   for some almost positive root~$\beta$. By Lemma~\ref{lem:LemA}, there is a $c$-singleton~$w$
   and a simple reflection~$s\in S$ such that~$\lr_s(w)=\beta$. But this implies that~$\rho$
   is an extremal ray of~$C(w)$.
\end{proof}

If~$w$ is a $c$-singleton, then~$C(w)\in\F_c^{(n)}$ is the maximal cone $w(D)$ which
is spanned by the rays $\{w(\rho_s)\,|\,s\in S\}$. The main result of this section is

\begin{thm}\label{thm:Explif_c}
   Let $\rho\in\F_c^{(1)}$. There is  a unique simple reflection~$s\in S$ and there is a $c$-singleton~$w$
   such that~$\rho=w(\rho_s)$ and $f_c(w(\rho_s))=\lr_s (w)$.
\end{thm}

\begin{proof}
   The uniqueness of $s\in S$ follows from the fact that any ray of
   the Coxeter fan is of the form $w(\rho_s)$ where $s\in S$ is
   uniquely determined (but not $w$!).

   The first claim follows directly from Lemma~\ref{lem:enoughsing}. We proceed by induction
   on the length of~$w$. If~$\ell(w)=0$ then~$w=e$. In particular,~$e$ is a $c$-singleton
   and~$s=es$ is $c$-sortable for any~$s\in S$. Fix some~$s\in S$. Since
   $\cl_c(e)=-\Delta=\{ -\a_t \,|\, t\in S\}$ and
   $\cl_c(s)=\{ -\a_t \,|\, t\in \la s\ra\}\cup\{\a_s\}$, we conclude $f_c(e(\rho_s))=-\alpha_s$
   as~$s(D)\subset C(s)$ and the rays of~$\F$ in~$s(D)\cap D$ are $\{\rho_t\,|\,t\in \la s\ra\}$.

   Suppose that~$\ell(w)>0$ and let $t\in S$ be the last simple reflection of the $c$-sorting
   word of~$w$. By Proposition~\ref{prop:Prefixes},~$wt$ is a $c$-singleton with
   $\ell(wt)<\ell(w)$. By induction, $f_c(wt(\rho_s))=\lr_s(wt)$ for some~$s\in S$.
   If~$s\neq t$ then $t\in W_{\la s\ra}$ and we conclude~$wt(\rho_s)=w(\rho_s)$\
   and~$\lr_s(wt)=\lr_s(w)$. Hence, we suppose~$s=t$. We have $C(w)\cap C(ws)=w(D)\cap ws(D)$,
   the extremal rays of this cone are $\{w(\rho_t)\,|\, t\in \la s\ra\}$, and their image
   under~$f_c$ is
   \[
     \cl_c(w)\cap\cl_c(ws)=\{\lr_t(w)\,|\, t\in \la s\ra\}=\cl_c(w)\setminus\{\lr_s(w)\}.
   \]
   So $f_c(w(\rho_s))=\lr_s(w)$.
\end{proof}

\section{Realizing generalized associahedra}\label{se:CamPoly}

\subsection{A general result}\label{sub:generalresult}
A fan has to satisfy some obvious conditions in order to have the
combinatorics of the normal fan of a polytope, in particular, the
fan has to be pointed and complete. This condition is far from
sufficient and in general it is quite hard to decide whether a given
fan is polytopal or not. Coarsenings or refinements of a (polytopal
or non-polytopal) fan may or may not be polytopal, examples are
easily derived from the non-polytopal fan
of~\cite[Example~$7.5$]{Ziegler} which is a pointed, complete and
simplicial fan in~$\R^3$. We aim for something less ambitious and
prove a criterion that implies that a given realization of a fan is
the normal fan of a polytope. Our notation is inspired
by~Section~\ref{se:CambFan}.

\medskip
Consider a pointed, complete, and simplicial fan~$\pf \subseteq
\R^n$ with $d$-dimensional cones~$\pf^{(d)}$. To $\rho\in\pf^{(1)}$
we associate a vector~$v_\rho$ such that~$\rho=\R_{\geq 0}[v_{\rho}]$.
Suppose that we are given a collection of positive real numbers $\lambda_\rho$,
one for each $\rho\in\pf^{(1)}$.  We then define
a hyperplane
\[
  H_\rho = \set{x\in\R^n}{\scalprod{x}{v_{\rho}}=\lambda_\rho}
\]
and a half space
\[
  \H_\rho:=\set{x\in\R^n}{\scalprod{x}{v_{\rho}}\leq\lambda_\rho}.
\]
We write
$\H_\rho^+$ if the inequality is strict. Since~$\pf$ is simplicial,
we have  for every maximal cone~$C\in\pf^{(n)}$ a point $x(C)$
defined by $\{x(C)\} := \bigcap_{\rho\in C^{(1)}}H_\rho$. Then
\begin{equation*}
  P := \operatorname{ConvexHull}
  \set{x(C)}{C\in\pf^{(n)}}\quad\textrm{and}\quad \widetilde{P}:= \bigcap_{\rho\in\pf^{(1)}}\H_\rho
\end{equation*}
are  well-defined polytopes of dimension at most $n$.

For instance the case of a $W$-permutahedron constructed from the Coxeter 
fan~$\F$ as explained in \S\ref{subsection:permutahedra} fits nicely in 
this context: $x(w(D))$ is by definition~$M(w)$ and the half spaces~$\H_\rho$ 
are precisely the half spaces~$\H^{\pmb a}_\rho$, for~$\rho\in\F^{(1)}$.  
In this case the two polytopes~$P$ and~$\widetilde{P}$ coincide.

Let~$C\in\pf^{(n)}$ and~$f\in C^{(n-1)}$ a $(n-1)$-dimensional face
of~$C$. An {\em outer normal of~$C$ relative to~$f$} is a vector~$v$
normal to~$f$, that is normal to the hyperplane spanned by~$f$, and
such that $C\subseteq \set{x\in\R^n}{\scalprod{x}{v}\leq 0}$.

Let $C_i,C_j\in\pf^{(n)}$ be two adjacent maximal cones in~$\pf$,
that is,  $C_i \cap C_j \in \pf^{(n-1)}$. A vector~$u$ is said to be
{\em pointing to $C_i$ from $C_j$} if there is an outer normal~$v$
of~$C_j$ relative to $C_i\cap C_j$ such that $\scalprod{u}{v}>0$. In
particular, observe that:
\begin{compactenum}[(i)]
   \item Any outer normal of~$C_j$ relative to~$C_i\cap C_j$ is
         pointing to~$C_i$ from~$C_j$;
   \item If $x_i\in C_i$ and $x_j\in C_j$ are points in the 
         interior of these cones, then the vector~$x_i-x_j$ 
         is pointing to~$C_i$ from~$C_j$;
   \item Any vector not contained in the span of $C_i\cap C_j$ 
         is either pointing to~$C_i$ from~$C_j$ or pointing 
         to~$C_j$ from~$C_i$.
\end{compactenum}
Notice that the vector~$x(C_i)-x(C_j)$ is a normal vector to
$C_i\cap C_j$, but not necessarily pointing to~$C_i$ from~$C_j$,
since~$x(C_i)$ is not necessarily a point in~$C_i$.

\begin{thm}\label{thm:general_result}
    Use the notation as above and suppose that~$x(C_i)-x(C_j)$ points to~$C_i$ 
    from~$C_j$ whenever $C_i,C_j\in\pf^{(n)}$ with $C_i \cap C_j \in \pf^{(n-1)}$.
    Then~$P=\widetilde P$ has (outer) normal fan $\mathcal N(P)=\pf$ and is 
    of dimension~$n$.
\end{thm}

\begin{rem}
   \textnormal{The hypothesis of Theorem~\ref{thm:general_result} is
               satisfied in (at least) three cases.}

   \textnormal{First, the case of $W$-permutahedra constructed from 
               the Coxeter fan~$\F$. Indeed, the points~$M(w)$ are 
               strictly in the cone~$w(D)$. So $M(w)-M(w')$ points 
               to~$w(D)$ from~$w'(D)$ whenever~$w(D)$ and~$w'(D)$ 
               are adjacent cones.}

   \textnormal{Second, the case of the cube constructed from the 
               fan~$\pf$ which is the skew coordinate hyperplane
               arrangement obtained from the hyperplanes which 
               bound the fundamental chamber in the Coxeter arrangement. 
               (Note that this fan corresponds to the usual construction 
               of the Boolean lattice as a quotient of weak order, via 
               the descent map, see~\cite{leconte}).}

   \textnormal{Third, as we shall show, the case of Cambrian fans described on the two next sections.}
\end{rem}

\begin{rem}
	\textnormal{A similar theorem, which includes the assumption that~$\pf$ should be a
                coarsening of a Coxeter fan, appears as~\cite[Theorem A.3]{kam}. That 
                theorem would suffice for our purposes, but we prefer to give the following
                independent proof of the more general theorem.  }
\end{rem}

\begin{proof}[Proof of Theorem~\ref{thm:general_result}]
    Since~$P$ and~$\widetilde P$ are convex polytopes, it is enough
    to show that the vertex set~$V(P)$ of~$P$ is equal to the
    vertex set~$V(\widetilde P)$ of~$\widetilde P$.

    Let us prove first that $P\subseteq \widetilde  P$.
    It suffices to prove $\scalprod{x(C)}{v_\rho}<\lambda_\rho$
    for~$C\in\pf^{(n)}$ and~$\rho\in\pf^{(1)}\setminus C^{(1)}$.

    Let~$C\in\pf^{(n)}$ and $\rho\in\pf^{(1)}\setminus C^{(1)}$. We
    will show that there is a finite sequence $C_0:=C,\ldots, C_k=C'$
    of maximal cones of~$\pf$ such that $\rho\subseteq C'$,
    $C_i\cap C_{i+1}\in \pf^{(n-1)}$ and~$v_\rho$ is pointing to~$C_{i+1}$ 
    from~$C_i$, for $0\leq i<k$.

    For~$x$ in~$C$, we write~$x+\rho$ for the half line
    $\set{x+\lambda v_\rho}{\lambda\geq 0}$ parallel to~$\rho$ and starting at~$x$.
    Write~$C_\rho$ for the union of all maximal cones of~$\pf$ that contain~$\rho$. 
    Since~$\pf$ is a pointed complete fan,~$C_\rho$ contains $n$-dimensional balls
    of arbitrary diameter centered at points of~$\rho$.  In particular,~$C_\rho$ 
    contains such a ball of diameter~$d$, where~$d$ is the distance between the 
    lines containing~$\rho$ and~$x+\rho$. So $(x+\rho)\cap C_\rho\not=\varnothing$ 
    for any point~$x\in C$. Hence there is a maximal cone~$C'$ of~$\pf$ such 
    that~$\rho$ is an extremal ray of~$C'$ and $(x+\rho)\cap C'\not=\varnothing$ 
    for any point~$x\in C$. For any~$x\in C$, the line segment between~$C$ and~$C'$ 
    on~$x+\rho$ determines a sequence of cones $C_0=C,\ C_1,\dots,\ C_p=C'$ of~$\pf$ 
    of arbitrary dimension, those are the cones that~$x+\rho$ meets between~$C$ and~$C'$
    in the natural order on the~$C_i$ induced by the order of points of~$x+\rho$
    given by the parametrization of this half line. From this point, we are looking 
    for a sequence $C=C_0,\ C_{0,1},\ C_1,\ C_{1,2},\dots,C_{k-1,k},\ C_{k}=C'$ such 
    that~$C_{i}$ is a maximal cone and $C_{i,i+1}=C_{i} \cap C_{i+1}$  is a cone of 
    codimension~$1$. Since the number of cones in~$\pf$ is finite, the number of cones 
    met by all possible half lines~$x+\rho$ for~$x\in C$ is finite. Since~$C$ is a full 
    dimensional cone, we may move~$x$ in~$C$ and then may assume that~$x+\rho$
    does not intersect any cone of~$\pf$ of codimension larger than~$1$. In other words, 
    there is a finite sequence $C_0=C,\ C_1,\dots,\ C_k=C'$ of maximal cones of~$\pf$ 
    such that~$C_{i,i+1}=C_i\cap C_{i+1}\in \pf^{(n-1)}$ and~$(x+\rho) \cap C_i\not =\varnothing$. 
    Pick~$y_i$ in the interior of~$C_i$ and in~$x+\rho$. So $y_{i+1}-y_i$ points 
    to~$C_{i+1}$ from~$C_i$. Since the cones $C_0,\dots,C_k$ have the same order as 
    the points on~$x+\rho$, the distance from~$x$ to~$y_i$ is strictly smaller than the 
    distance from~$x$ to~$y_{i+1}$. This means the vector $y_{i+1}-y_i= \kappa v_\rho$ 
    with~$\kappa >0$. Hence~$v_\rho$ points to~$C_{i+1}$ from~$C_i$.

    Now, we consider the piecewise linear path from~$x(C_0)$ to~$x(C_{k})$ that traverses
    from~$x(C_i)$ to~$x(C_{i+1})$. Since $x(C_{i+1})-x(C_i)$ points to~$C_{i+1}$ 
    from~$C_i$, the vector~$x(C_{i+1})-x(C_i)$ is an outer normal to~$C_i$ relative 
    to~$C_i\cap C_{i+1}$, and since~$v_\rho$ points to~$C_{i+1}$ from~$C_i$ for $0\leq i<k$, 
    we conclude that~$\scalprod{x(C_{i+1})-x(C_i)}{v_\rho}>0$. Hence
    \[
      \scalprod{x(C_0)}{v_\rho}<\ldots<\scalprod{x(C_k)}{v_\rho}=\lambda_\rho.
    \]
    This proves $P\subseteq \widetilde P$.

    For any vertex~$x(C)$ of~$P$ we know that~$x(C)$ is a vertex of~$\widetilde P$ since~$x(C)$ is defined as
    the intersection of hyperplanes that bound half spaces defining~$\widetilde P$ and~$P \subseteq \widetilde P$.
    Moreover, for any maximal cone~$C\in\pf$, the point~$x(C)$ must be on the boundary of~$P$ for that reason.
    But now it follows that the points~$x(C)$ are in convex position because~$\pf$ is simplicial and
    $\bigcap_{\rho\in C^{(1)}}H_\rho$ is a singleton. Hence
    \[
      \set{x(C)}{C\in\pf^{(n)}}=V(P)\subseteq V(\widetilde P).
    \]
    The rays~$\rho\in C^{(1)}$ span the outer normal cone of~$x(C)\in P$ for every maximal cone~$C\in\pf$.
    The hyperplanes~$H_\rho$, $\rho \in C^{(1)}$, bound the half spaces of~$x(C) \in \widetilde P$, so the normal
    cones agree. But the union of these cones equals the complete fan~$\pf$, so~$\widetilde P$ does not have any
    additional vertices.

    The claim that $\operatorname{dim}(P)=n$ follows from the fact that $\lambda_\rho>0$ for all $\rho\in\pf^{(1)}$:
    a neighborhood of $0$ is contained in~$P$.
\end{proof}

\subsection{Realizations of generalized associahedra}
We apply Theorem~\ref{thm:general_result} to show how $c$-Cambrian
fans~$\F_c$ and associahedra~$\Ass^{\pmb a}_c(W)$ relate. The
associahedron is described as intersection of certain half spaces of
the permutahedron~$\Perm^{\pmb a}(W)$ determined by the rays
of~$\F_c$ and the common vertices of~$\Ass^{\pmb a}_c(W)$
and~$\Perm^{\pmb a}(W)$ are characterized in terms of
$c$-singletons. The proof of Theorem~\ref{thm:main_theorem} is
deferred to Section~\ref{subsec:proof}.

\begin{thm}\label{thm:main_theorem}
   Let~$c$ be a Coxeter element of~$W$ and choose a point~$\pmb a$ in the
interior of the
fundamental chamber $D$,
   to fix a realization of the permutahedron~$\Perm^{\pmb a}(W)$.
   \begin{compactenum}[(i)]
      \item The polyhedron
            \[
              \Ass^{\pmb a}_{c}(W)=\bigcap_{\rho\in\F_c^{(1)}} \H^{\pmb a}_{\rho}
            \]
            is a simple polytope of dimension~$n$ with $c$-Cambrian fan~$\F_c$ as normal fan.
      \item The vertex sets~$V(\Ass^{\pmb a}_{c}(W))$ and~$V(\Perm^{\pmb a}(W))$ satisfy
            \[
              V(\Ass^{\pmb a}_{c}(W))\cap V(\Perm^{\pmb a}(W))
                    =\{ M(w) \, | \, w \text{ is a $c$-singleton}\}.
            \]
\end{compactenum}
\end{thm}

The first statement implies that the facet-supporting half spaces of the associahedron form a
subset of the facet-supporting half spaces of the permutahedron. We mentioned this in the
introduction in the context of $c$-admissible half spaces. A facet-supporting half
space~$\H^{\pmb a}_{\rho}$ of the permutahedron~$\Perm^{\pmb a}(W)$ is {\em $c$-admissible}
if~$M(w)\in H^{\pmb a}_{\rho}$ for some $c$-singleton~$w$. We rephrase the first statement
as follows:

\begin{cor}\label{cor:Intro}
   The associahedron~$\Ass^{\pmb a}_{c}(W)$ is the intersection of all $c$-admissible halfspaces
   of~$\Perm^{\pmb a}(W)$.
\end{cor}

\noindent
We illustrate these results with a basic example.

\begin{figure}[b]
      \begin{center}
      \begin{minipage}{0.90\linewidth}
         \begin{center}
         \begin{overpic}
            [width=\linewidth]{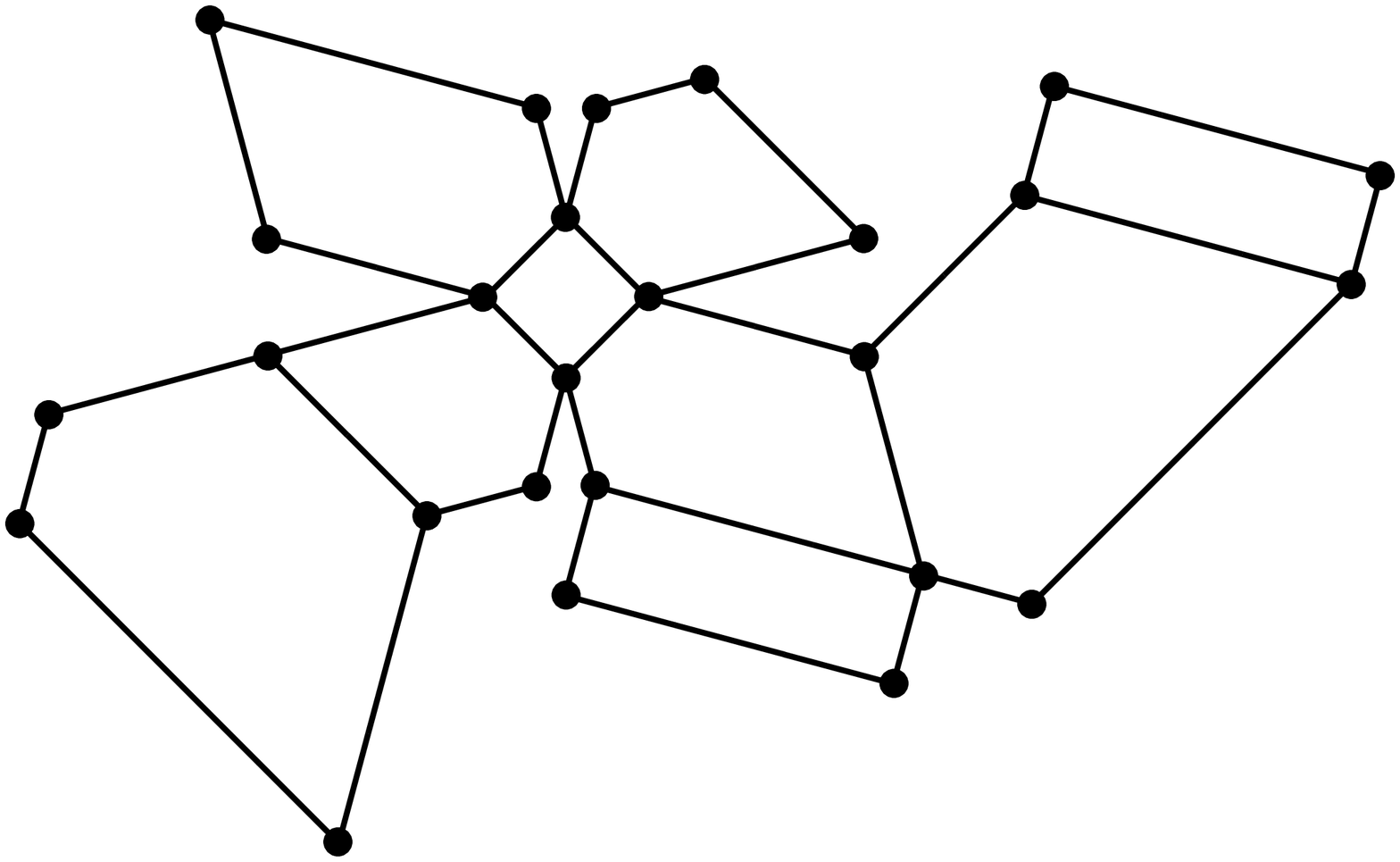}
            \put(29.5,24){\tiny $A$}
            \put(42,21.5){\tiny $A$}
            \put(36.5,28){\tiny $B$}
            \put(43.5,28){\tiny $B$}
            \put(42,32){\tiny $C$}
            \put(34.7,39.3){\tiny $D$}
            \put(44.5,35){\tiny $E$}
            \put(38,42.3){\tiny $F$}
            \put(36,46){\tiny $G$}
            \put(43.5,46.5){\tiny $G$}
            \put(69,47){\tiny $G$}
            \put(47,47.3){\tiny $H$}
            \put(65.5,40.5){\tiny $H$}
            \put(58,17){\tiny $1$}
            \put(66.5,21.5){\tiny $1$}
            \put(27,8.5){\tiny $1$}
            \put(62,23){\tiny $2$}
            \put(59,33){\tiny $3$}
            \put(55,41.2){\tiny $3$}
            \put(82.5,37){\tiny $4$}
            \put(11.5,24.5){\tiny $4$}
            \put(85,43){\tiny $5$}
            \put(12.5,28.8){\tiny $5$}
            \put(22.5,50.5){\tiny $5$}
            \put(22.5,31.5){\tiny $6$}
            \put(25,42){\tiny $6$}
            \put(77,1){\begin{overpic}[width=4cm]{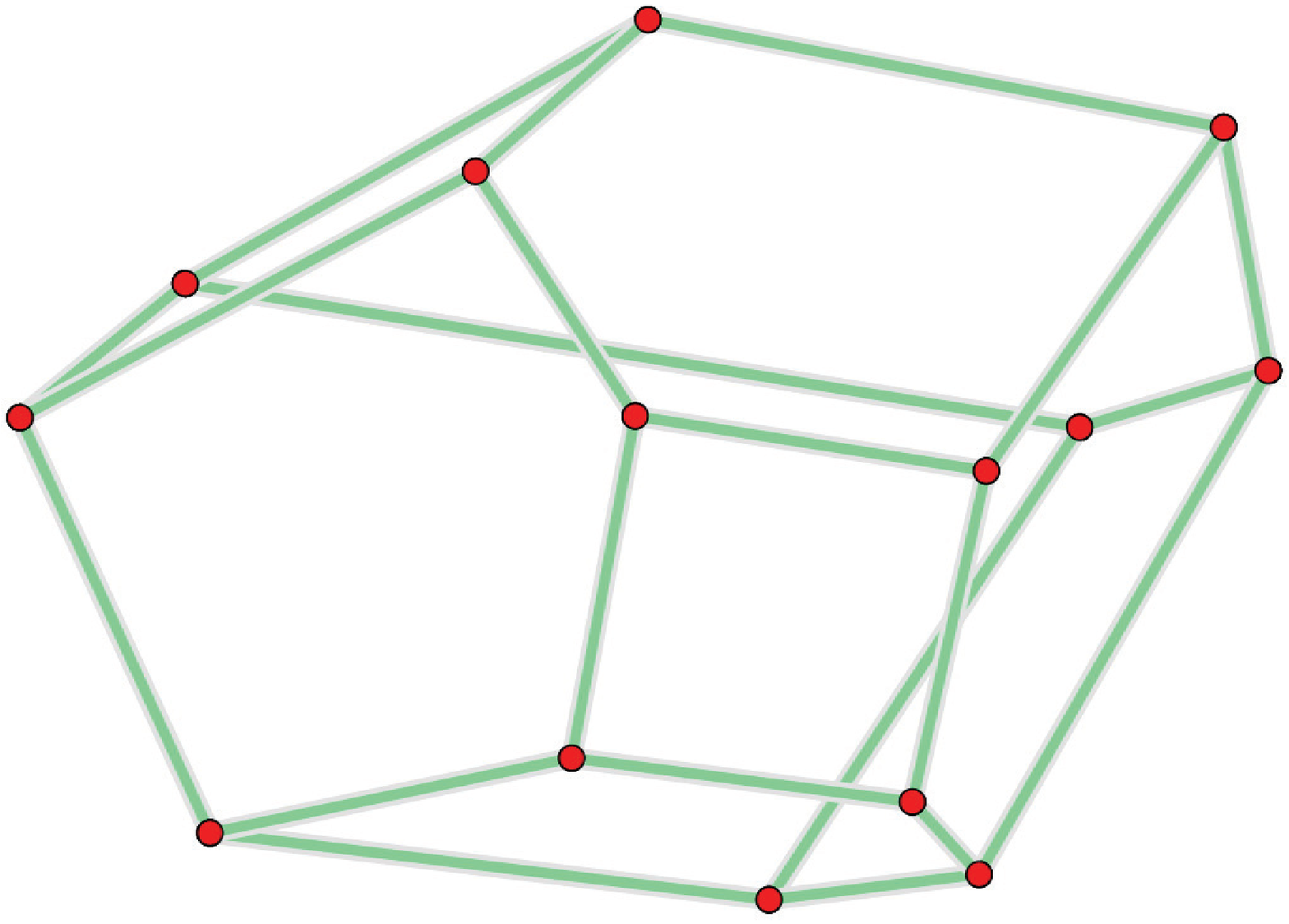}
                           \put(53,2){\tiny $A$}
                           \put(68,5){\tiny $B$}
                           \put(64.8,12.5){\tiny $C$}
                           \put(43.5,16.5){\tiny $D$}
                           \put(60.5,28.5){\tiny $E$}
                           \put(39.5,33.5){\tiny $F$}
                           \put(37.7,49.7){\tiny $G$}
                           \put(43.8,63){\tiny $H$}
                           \put(73.5,31){\tiny $1$}
                           \put(86.5,36){\tiny $2$}
                           \put(83.5,55){\tiny $3$}
                           \put(13,45){\tiny $4$}
                           \put(3,33.5){\tiny $5$}
                           \put(16,6){\tiny $6$}
                        \end{overpic}}
            \put(20,23){$\rho_1$}
            \put(47,43){$\rho_8$}
            \put(49,31){$\rho_5$}
            \put(67,33){$\rho_6$}
            \put(76,44){$\rho_7$}
            \put(39.5,37.5){$\rho_4$}
            \put(49,21){$\rho_3$}
            \put(32,32){$\rho_2$}
            \put(30,45){$\rho_9$}
         \end{overpic}
         \end{center}
         \caption[]{An unfolding of the associahedron~$\Ass^{\pmb a}_{c}(S_4)$ with~$c=s_1s_2s_3$.
                    The $2$-faces are labelled by~$\rho_i \in \F_c^1$ for the facet-defining
                    hyperplane~$H_{\rho_i}^{\pmb a}$.}
         \label{fig:unfolded_a3_asso_123}
      \end{minipage}
      \end{center}
\end{figure}

\begin{expl}
   \textnormal{The first statement of Theorem~\ref{thm:main_theorem} claims that the intersection of a subset of the
        half spaces~$\H^{\pmb a}_{\rho}$ of~$\Perm^{\pmb a}(W)$ yields a generalized
        associahedron~$\Ass^{\pmb a}_c(W)$ if we restrict to half spaces such that~$\rho$ is a
        ray of the $c$-Cambrian fan~$\F_c$. Figures~\ref{fig:unfolded_a3_asso_123}
        and~\ref{fig:unfolded_a3_asso_213} illustrate this for~$W=S_4$ generated by $
        S=\{s_1,s_2,s_3\}$. We use the following conventions:
        The point~$\pmb a$ used to fix a realization of~$\Perm^{\pmb a}(W)$ is labeled~$A$.
        A facet of the associahedron~$\Ass_c^{\pmb a}(W)$ is labeled by the ray~$\rho_j \in \F_c$
        that is perpendicular to that facet. Recall that each ray~$\rho$ can be written
        as~$w(\rho_{s_i})$ for some (non-unique) $c$-singleton~$w$ and some (unique) simple
        reflection~$s_i$ by Lemma~\ref{lem:enoughsing}. If the Coxeter element~$c=s_1s_2s_3$ is
        chosen (Figure~\ref{fig:unfolded_a3_asso_123}), then we can express the ray~$\rho_i\in\F_c$
        that corresponds to the ($c$-admissible) half space~$\H^{\pmb a}_{\rho_i}$ as follows:}
        \begin{align*}
           \rho_1 &=e(\rho_{s_1}),\\
           \rho_2 &=e(\rho_{s_3})=s_1(\rho_{s_3})=s_1s_2(\rho_{s_3})=s_1s_2s_1(\rho_{s_3}),\\
           \rho_3 &=e(\rho_{s_2})=s_1(\rho_{s_2}),\\
           \rho_4 &= s_1s_2(\rho_{s_2})= s_1s_2s_3(\rho_{s_2}) = s_1s_2s_1(\rho_{s_2})= s_1s_2s_3s_1(\rho_{s_2}),\\
           \rho_5 &= s_1(\rho_{s_1})=s_1s_2(\rho_{s_{1}})=s_1s_2s_3(\rho_{s_1}),\\
           \rho_6 &= s_1s_2s_3s_1s_2s_1(\rho_{s_1}), \\
           \rho_7 &= s_1s_2s_3s_1s_2(\rho_{s_2})= s_1s_2s_3s_1s_2s_1(\rho_{s_2}),\\
           \rho_8 &= s_1s_2s_3(\rho_{s_3}) = s_1s_2s_3s_1(\rho_{s_3}) = s_1s_2s_3s_1s_2(\rho_{s_3})
                   = s_1s_2s_3s_1s_2s_1(\rho_{s_3}),\text{ and}\\
           \rho_9 &=s_1s_2s_1(\rho_{s_1})=s_1s_2s_3s_1(\rho_{s_1})=s_1s_2s_3s_1s_2(\rho_{s_1}).
        \end{align*}
    \noindent
    \begin{figure}
          \begin{center}
          \begin{minipage}{0.90\linewidth}
             \begin{center}
             \begin{overpic}
                [width=\linewidth]{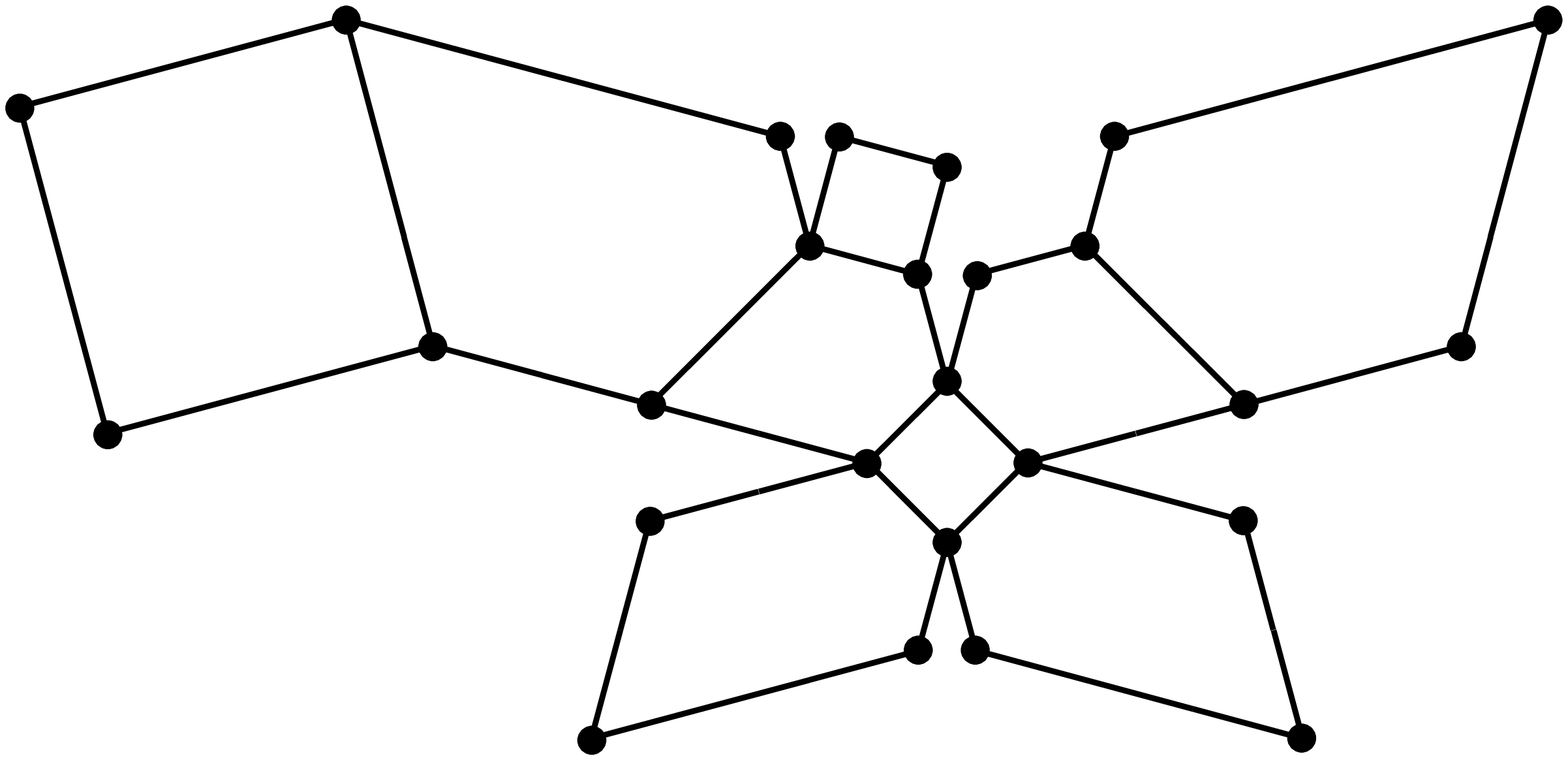}
                \put(55,18){\tiny $A$}
                \put(62,18){\tiny $A$}
                \put(11.5,31){\tiny $A$}
                \put(60.5,22){\tiny $B$}
                \put(53,25){\tiny $C$}
                \put(63.5,25){\tiny $D$}
                \put(60.5,32){\tiny $E$}
                \put(62,36.5){\tiny $F$}
                \put(55,36.5){\tiny $F$}
                \put(48,39.7){\tiny $G$}
                \put(60.5,44){\tiny $H$}
                \put(68.5,39.5){\tiny $H$}
                \put(47.5,44.5){\tiny $I$}
                \put(53.5,44){\tiny $I$}
                \put(70,44.5){\tiny $I$}
                \put(8,46.5){\tiny $1$}
                \put(77.5,13.5){\tiny $1$}
                \put(87,35.5){\tiny $1$}
                \put(74.5,22){\tiny $2$}
                \put(76.5,33){\tiny $2$}
                \put(91,50){\tiny $3$}
                \put(23.5,50){\tiny $3$}
                \put(40.5,13.5){\tiny $4$}
                \put(28,35){\tiny $4$}
                \put(43,22){\tiny $5$}
                \put(44.5,31.5){\tiny $5$}
                \put(77,1){\begin{overpic}
                                [width=3.5cm]{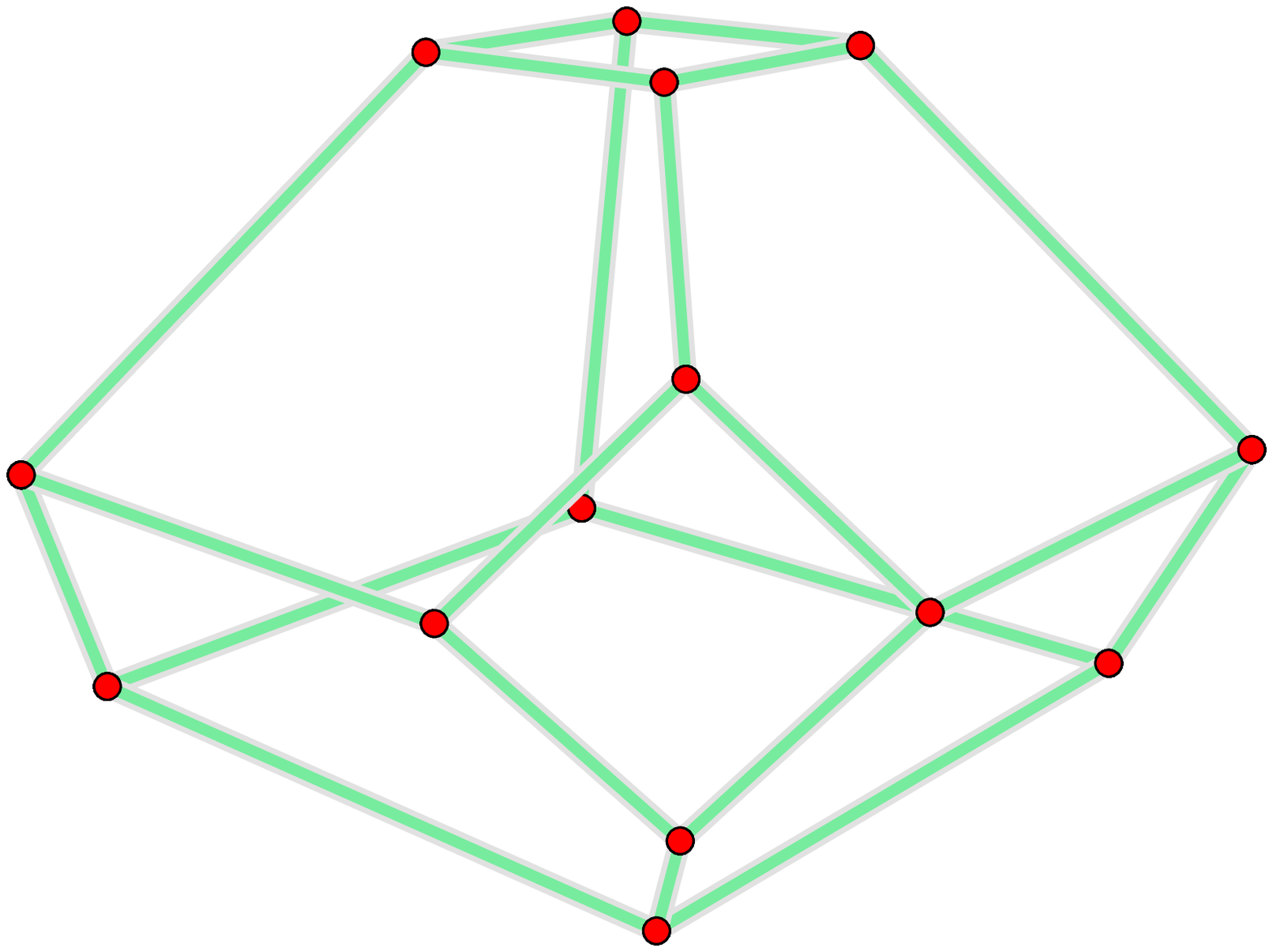}
                                \put(58,0){\tiny $A$}
                                \put(53,8){\tiny $B$}
                                \put(39,18){\tiny $C$}
                                \put(73,18){\tiny $D$}
                                \put(62,39){\tiny $E$}
                                \put(61,52){\tiny $F$}
                                \put(37,59){\tiny $G$}
                                \put(74,59){\tiny $H$}
                                \put(55,63.5){\tiny $I$}
                                \put(87,17){\tiny $1$}
                                \put(97,31){\tiny $2$}
                                \put(52,25){\tiny $3$}
                                \put(14,29){\tiny $4$}
                                \put(21,15){\tiny $5$}
                \end{overpic}}
                \put(17,41){$\rho_4$}
                \put(36,41){$\rho_5$}
                \put(51,33.4){$\rho_7$}
                \put(65,33.5){$\rho_8$}
                \put(77,41){$\rho_9$}
                \put(58,27.6){$\rho_2$}
                \put(48,20){$\rho_1$}
                \put(67,20){$\rho_3$}
                \put(54,42){$\rho_6$}
             \end{overpic}
             \end{center}
             \caption[]{An unfolding of the associahedron~$\Ass^{\pmb a}_{c}(S_4)$ with~$c=s_2s_1s_3$.
                       The $2$-faces are labeled by~$\rho=(w,s) \in \F_c^1$ for the facet-defining
                       hyperplane~$H_\rho^{\pmb a}$.}
             \label{fig:unfolded_a3_asso_213}
          \end{minipage}
          \end{center}
    \end{figure}
    \noindent
   \textnormal{The claim of the second statement of Theorem~\ref{thm:main_theorem} is that the
         common vertices of~$\Perm^{\pmb a}(W)$ and~$\Ass^{\pmb a}_c(W)$ are the
         points $M(w)=w(M(e))$ for $w$ a $c$-singleton. It is straightforward to
         verify this claim directly if~$c=s_1s_2s_3$ in Figure~\ref{fig:unfolded_a3_asso_123}:
         The common vertices of~$\Ass^{\pmb a}_{c}(W)$ and~$\Perm^{\pmb a}(W)$ are labeled~$A$
         through~$H$ and we have}
         \begin{alignat*}{3}
            A &= M(e),               &\qquad B &= M(s_1),                &\qquad C &= M(s_1s_2),\\
            D &= M(s_1s_2s_1),       &\qquad E &= M(s_1s_2s_3),          &\qquad F &= M(s_1s_2s_3s_1),\\
            G &= M(s_1s_2s_3s_1s_2), &\qquad H &= M(s_1s_2s_3s_1s_2s_1). &         &
          \end{alignat*}
    \textnormal{If the Coxeter element is~$c=s_2s_1s_3$ (Figure~\ref{fig:unfolded_a3_asso_213}) then we
                have the following list of expressions for~$\rho_i\in\F_c$ (we do not list all possible
                expressions for~$\rho_i$):}
         \begin{alignat*}{3}
            \rho_1 &= e(\rho_{s_3}),    &\qquad \rho_2 &= s_2(\rho_{s_2}),
                                        &\qquad \rho_3 &= e(\rho_{s_1}), \\
            \rho_4 &= e(\rho_{s_2}),    &\qquad \rho_5 &= s_2s_1s_3s_2s_3(\rho_{s_3})
                                        &\qquad \rho_6 &= s_2s_1s_3s_2s_1s_3(\rho_{s_2}), \\
            \rho_7 &=s_2s_1(\rho_{s_1}),&\qquad \rho_8 &= s_2s_3(\rho_{s_3}),
                                        &\qquad \rho_9 &= s_2s_1s_3s_2s_1(\rho_{s_1}).
         \end{alignat*}
    \noindent
    \textnormal{The common vertices of the permutahedron and associahedron are labeled~$A$ through~$I$
         and we have}
         \begin{alignat*}{3}
            A &= M(e),               &\qquad B &= M(s_2),                &\qquad C &= M(s_2s_1),\\
            D &= M(s_2s_3),          &\qquad E &= M(s_2s_1s_3),          &\qquad F &= M(s_2s_1s_3s_2),\\
            G &= M(s_2s_1s_3s_2s_3), &\qquad H &= M(s_2s_1s_3s_2s_1), &\qquad I &= M(s_2s_1s_3s_2s_1s_3).
          \end{alignat*}
\end{expl}

\subsection{Proof of Theorem~\ref{thm:main_theorem}}\label{subsec:proof}
Let~$c \in W$ be a Coxeter element and let $\pmb a$ be in the interior
of~$D$. The proof of Theorem~\ref{thm:main_theorem} is based on 
Theorem~\ref{thm:general_result}. We use the same notations as 
in~\S\ref{sub:generalresult} applied to the $c$-Cambrian fan~$\F_c$ and
denote by~$x(C)$ the intersection point of the hyperplanes 
$H^{\pmb a}_\rho$ for $\rho$ the extremal rays of a maximal 
cone~$C$ of~$\F_c$.

Let~$w$ and~$w^\prime$ be distinct $c$-sortable elements such that
the associated maximal cones~$C:=C(w)$ and~$C^\prime:=C(w^\prime)$
of the Cambrian fan~$\F_c$ intersect in a cone of codimension~$1$.
So either~$w$ is a cover of~$w'$ or~$w'$ is a cover of~$w$ in the
lattice of $c$-sortable elements. Without loss of generality, we 
may assume that~$w$ is a cover of~$w^\prime$. To meet the requirement 
of Theorem~\ref{thm:general_result}, we have to prove that the 
vector $x(C)-x(C')$ points to~$C$ from~$C'$.

\medskip
We use the following notations. 
Set~$\mathcal R := \set{w(v_s)}{w\in W, s\in S}$ so that the set of rays
of the Coxeter fan is 
$\R_{\geq 0}\mathcal R := \set{\lambda v}{\lambda \geq 0, v \in \mathcal R}$. 
For~$v\in \mathcal R$, we write $\nu_v:=\scalprod{\pmb a}{v_s}=\scalprod{M(w)}{v}>0$ 
which depends only on~$s$.

The intersection~$C\cap C'$ is contained in a hyperplane~$H_t$ for
some reflection~$t\in W$ since~$\F_c$ is a coarsening of the
Coxeter fan~$\F$. We now show which of the two roots associated 
to $H_t$ is an outer normal to $C'$ relative to $C'\cap C$.

\begin{lem}\label{lem:alpha_is_negative}
   Let $w, w^\prime \in W$ be $c$-sortable elements such that~$w$ is a cover 
   of~$w^\prime$ in the lattice of $c$-sortable elements. $C(w)\cap C(w^\prime)$ 
   is an $n-1$-dimensional cone of~$\F_c$, and suppose it lies in~$H_t$
   for some reflection~$t$. Let the roots associated to~$t$ be~$\pm \beta$.
   If~$\beta$ is an outer normal to~$C'$ relative to~$C'\cap C$, then~$\beta$ 
   is a negative root.
\end{lem}

\begin{proof}
    Let $\widetilde w \in (\pid)^{-1}(w^\prime)$ such that~$w$ is a cover 
    of~$\widetilde w$ in the right weak order. Then 
    $w(D) \cap \widetilde w(D)\subset H_t$ is an $(n-1)$-dimensional cone 
    of the Coxeter fan~$\F$. Since~$\beta$ is an outer normal for~$C(w')$ 
    relative to $C(w)\cap C(w')$, it is an outer normal for~$\widetilde w(D)$ 
    with respect to $\widetilde w(D)\cap w(D)$.

    \textbf {Case $\pmb 1$:}
    Suppose that $w'=\widetilde w=e$. Then~$w=s$ for some $s\in S$. The hyperplane
    dividing~$w(D)$ from~$D$ is~$H_s$, perpendicular to~$\alpha_s$. $D$ lies
    on the side of~$H_s$ having positive inner product with~$\alpha_s$. Thus
    the outer normal~$\beta=-\alpha_s$ is a negative root.

    \textbf {Case $\pmb 2$:}
    Suppose $\widetilde w\ne e$. Then $\widetilde w^{-1}w=s\in S$ since~$w$ 
    covers~$\widetilde w$ in right weak order. By the previous case, 
    $\widetilde w(-\alpha_s)$ is an outer normal. Since 
    $\ell(w)=\ell(\widetilde ws)>\ell(\widetilde w)$, we have that
    $\beta=\widetilde w(-\alpha_s)$ is a negative root, as desired.
\end{proof}

We have that $C\cap \R_{\geq 0}\mathcal R = \{ \rho_{u_1}, \ldots , \rho_{u_p}\}$ 
and since~$\F_c$ is simplicial, we may assume that the extremal rays of~$C$ are 
the first~$n:=|S|$ rays. Similarly, we may assume 
$\{ \rho_{u^\prime_1}, \rho_{u_2}, \ldots ,\rho_{u_n}\}$ are the extremal rays 
of~$C^\prime$. Hence, $H_t$ is spanned by $\{ u_2, \ldots , u_n\}$. As we have
$x(C):=\bigcap_{i=1}^{n}H^{\pmb a}_{\rho_{u_i}}$ and
$x(C)-x(C^\prime)\in \bigcap_{i=2}^{n}H^{\pmb a}_{\rho_{u_i}}$, we
conclude $x(C)-x(C^\prime) = \mu\beta$ for some $\mu \in \R$ and $\beta$
a negative root. Thus, $x(C)-x(C')$ is pointing to $C$ from $C'$ if and
only if $x(C)-x(C')=\mu\beta$ with $\mu>0$.

\begin{lem} \label{lem:equivalent_condition}
    Let $w, w^\prime \in W$ be $c$-sortable elements such that $w$ covers~$w^\prime$ in the lattice of
    $c$-sortable elements and $C(w)\cap C(w^\prime)\subset H_t$ is an $(n-1)$-dimensional cone of~$\F_c$ for a
    reflection~$t$. Let the extremal rays of~$C:=C(w)$ and~$C^\prime:=C(w^\prime)$ be generated
    by~$\{ u_1, \ldots , u_n\}$ and $\{ u_1^\prime, u_2, \ldots , u_n\}$ and suppose
    $u_1 + u_1^\prime = \sum_{i=2}^{n} b_iu_i\in H_t$ with $b_i \geq 0$.
    Then the following statements are equivalent:
    \begin{compactenum}[(i)]
       \item $x(C)-x(C')$ is pointing to $C$ from $C'$;
       \item $x(C)-x(C')=\mu\beta$ with $\beta\in\Phi^-$ and $\mu>0$;
       \item $\scalprod{x(C)-x(C^\prime)}{u_1}>0$;
       \item $\nu_{u_1} + \nu_{u_1^\prime} > \sum_{i=2}^{n}b_{i}\nu_i$.
    \end{compactenum}
\end{lem}
\begin{proof}
   The first equivalence follows from Lemma~\ref{lem:alpha_is_negative} and the preceding discussion.

   As $u_1\in C$ and~$C$ is spanned by the vectors in $C\cap C'$ together with~$u_1$, we 
   have $\scalprod{\beta}{u_1}>0$ if and only if~$\beta$ is an outer normal of~$C'$ relative
   to $C\cap C'$. This shows the second equivalence.

   The last equivalence follows from
   \[
     \scalprod{x(C)-x(C^\prime)}{u_1} = \nu_{u_1} - \scalprod{x(C^\prime)}{-u_1^\prime + \sum_{i=2}^{n}b_iu_i}
                                      = \nu_{u_1} + \nu_{u_1^\prime} - \sum_{i=2}^{n}b_i\nu_{u_i}.
   \]
\end{proof}

\medskip $ $\\
\medskip
\noindent 
We apply Theorem~\ref{thm:general_result} to conclude:
\begin{quote}
    If one of the equivalences in Lemma~\ref{lem:equivalent_condition} is achieved for all pairs of
    adjacent cones in $\F_c$, then~$\F_c$ is the outer normal fan of~$\Ass^{\pmb a}_c(W)$ which proves 
    Theorem~\ref{thm:main_theorem}.
\end{quote}
This will be done in  Lemma~\ref{lem:final_lemma} below.

\medskip
We first consider the special case that $e$ is covered by $s\in S$ and there is a reduced expression 
for~$c$ that starts with~$s$, that is~$s$ is \emph{initial} in~$c$.

\begin{lem}\label{lem:s_initial_for_c}
    Let $s\in S$ be initial in~$c$.
    Then $C(e) \cap C(s) \subseteq H_s$, $u_1^\prime = v_s$, $u_1 = s(v_s)$, and
    \[
      u_1^\prime + u_1 = \sum_{r\neq s}b_ru_r \in H_s
    \]
    with $b_r = -2\tfrac{\scalprod{\alpha_s}{\alpha_r}}{\scalprod{\alpha_s}{\alpha_s}}\geq 0$. Moreover,
    $\nu_{u_1^\prime} + \nu_{u_1} > \sum_{r\neq s}b_r\nu_{u_r}$.
\end{lem}

\begin{proof}
    Without loss of generality, assume $S=\{ s_1 , \ldots , s_n\}$ and $s = s_1$. Since~$s\in S$ is initial,
    $s$ is a $c$-singleton. The maximal cones $C(e)$ and $C(s)$ of~$\F_c$ are therefore maximal cones of the
    Coxeter fan with extremal rays generated by $\{ v_{s_1}, \ldots, v_{s_n}\}$
    and $\{ s(v_{s_1}), v_{s_2}, \ldots , v_{s_n}\}$. Since
    $\alpha_r = \sum_{i=1}^{n}\scalprod{\alpha_r}{\alpha_i}v_i$,
    we have
    \[
      s_1(v_{s_1})
          = v_{s_1}
            - 2\tfrac{ \scalprod{\alpha_{s_1}}{v_{s_1}} }
                     { \scalprod{\alpha_{s_1}} {\alpha_{s_1}} } \alpha_{s_1}
          = -v_{s_1}
            + \sum_{i=2}^n
                  \left(
                      -2\tfrac{ \scalprod{\alpha_{s_1}}{\alpha_{s_i}} }{ \scalprod{\alpha_{s_1}}{\alpha_{s_1}} }
                  \right)
                  v_{s_i}.
    \]
    In particular, $s(v_{s_1}) + v_{s_1} \in H_{s_1}$ and $\scalprod{\alpha_{s_1}}{\alpha_{s_i}}\leq 0$ for
    $s_1 \neq s_i$.
    As $\pmb a$ is a vertex of~$Perm^{\pmb a}(W)$, we conclude
\[    \nu_{s_1(u_1)} = \scalprod{\pmb a}{v_{s_1}}
                     > \scalprod{\pmb a}{s_1(v_{s_1})}
                     = -\nu_{s_1(v_{s_1})}+\sum_{i=2}^nb_i\nu_{s_i}
\]
\end{proof}

Some terminology and results due to N.~Reading and D.~Speyer are needed to prove Lemma~\ref{lem:final_lemma}
(and therefore to finish the proof of Theorem~\ref{thm:main_theorem}). To distinguish objects related to a
Cambrian fan with respect to different Coxeter elements, we use the Coxeter element as index. For example,
if we use the Coxeter element~$scs$ instead of~$c$, then~$C_{scs}(w)$ denotes the maximal cone that
corresponds to the $scs$-sortable element~$w$. If $s\in S$ is initial in~$c$ then $\F_{sc}$ is the
$sc$-Cambrian fan for the Coxeter element $sc$ of~$W_{\la s\ra}$.

\begin{lem}[{\cite[Lemmas~4.1,~4.2]{reading4}}] \label{lem:w_c_sortable}
   Let~$c$ be a Coxeter element and~$s$ initial in~$c$.
   \begin{compactenum}[(i)]
      \item Let~$w\in W$ such that $\ell(sw)<\ell(w)$. Then~$w$ is $c$-sortable if and only
            if~$sw$ is $scs$-sortable.
      \item Let~$w\in W$ such that $\ell(sw)>\ell(w)$. Then~$w$ is $c$-sortable if and only if
            $w\in W_{\la s\ra}$ and~$w$ is $sc$-sortable.
   \end{compactenum}
\end{lem}

Note that $\ell(sw)<\ell(w)$ if and only if the chamber~$w(D)$ of the Coxeter arrangement 
corresponding to~$w$ lies above~$H_s$. In this case, the maximal cone~$C(w)$ of~$\F_c$ is 
above~$H_s$ because~$w$ is minimal in its 
fibre~$(\pi^c_\downarrow)^{-1}\pi^c_\downarrow(w)=[\pi^c_\downarrow(w),\pi^\uparrow_c(w)]$
for the $c$-Cambrian congruence. On the other hand, if $\ell(sw)>\ell(w)$, then~$w(D)$ is 
below~$H_s$ in the Coxeter arrangement. In this case, we know that the maximum element of 
the fibre~$[\pi^c_\downarrow(w),\pi^\uparrow_c(w)]$ for~$w$, and thus all of~$C(w)$, is 
below~$H_s$, by~\cite[Lemma 4.11]{reading4}. It follows that the hyperplane~$H_s$ separates 
the cones of~$\F$ into two families and it never intersects a maximal cone of~$\F_c$ in its 
interior. For~$\rho\in\F^{(1)}_c$ we define
\[
  \zeta_s(\rho):=\begin{cases}
                     s(\rho) & \text{if $\rho\ne \rho_s$},\\
                     -\rho_s & \text{otherwise}.
                 \end{cases}
\]
We abuse notation and consider~$\zeta_s$ also as a map on the set of
vectors generating the rays~$\F_c^{(1)}$. The following lemma is a
consequence of~\cite[Lemma 6.5]{reading4} and~\cite[Theorem~1.1]{reading4}. 
Compare also the comments after~\cite[Corollary~7.3]{reading4}, from which 
the last statement is taken.

\begin{lem}[{\cite{reading4}}]\label{lem:six}
   Let~$s\in S$ be initial in the Coxeter element~$c$. If $\rho_1,\dots,\rho_n$ are the extremal
   rays of the maximal cone~$C(w)\in\F_c$ then $\zeta_s(\rho_1),\dots,\zeta_s(\rho_n)$ are the
   extremal rays of a maximal cone of~$\F_{scs}$.  If~$\ell(sw)<\ell(w)$, then these extremal rays
   are the extremal rays of the maximal cone~$C(sw)$ that corresponds to the $scs$-sortable
   element~$sw$.
\end{lem}

Before we finish the proof of Theorem~\ref{thm:main_theorem} with Lemma~\ref{lem:final_lemma} we make
an observation that will be useful also in Section~\ref{subsec:integer_coords}.

\begin{lem}\label{lem:cones}
   Let $c$ be a Coxeter element, $s\in S$ initial in~$c$, and $w\in W_{\la s \ra}$
   $sc$-sortable. Then the maximal cone~$C(w)\in \F_c$ is spanned
   by~$C_{sc}(w) \in \F_{sc}$ and the ray~$\rho_s\in \F_{c}$.
\end{lem}
\begin{proof}
   The ray~$\rho_s$ is the unique ray of~$\F_c$ that is strictly below~$H_s$ by~\cite[Lemma 6.3]{reading4}.
   From Lemma~\ref{lem:w_c_sortable} it follows that~$C(w)$ is below~$H_s$ and has~$\rho_s$ as an extremal ray.
   Hence~$C(w)$ is spanned by~$\rho_s$ and a maximal cone~$E(w):=C(w)\cap H_s\in\A_{\la s \ra}$.

   Now consider all inversions~$t$ of~$w$, that is, all reflections~$H_t\in~\A$ such that~$C(w)$
   is above~$H_t$. Since $w\in W_{\la s \ra}$ we conclude~$t \in W_{\la s \ra}$. Hence, the
   inversions of~$E(w)$ and~$C_{sc}(w)$ coincide and~$E(w)=C_{sc}(w)$.
\end{proof}

\begin{lem}\label{lem:final_lemma}
    Let $w, w^\prime \in W$ be $c$-sortable elements such that $w$ covers~$w^\prime$ in the lattice of
    $c$-sortable elements and $C(w)\cap C(w^\prime)\subset H_t$ is an $(n-1)$-dimensional cone of~$\F_c$
    for a reflection~$t$. Let the extremal rays of~$C:=C(w)$ and~$C^\prime:=C(w^\prime)$ be generated
    by~$\{ u_1, \ldots , u_n\}$ and $\{ u_1^\prime, u_2, \ldots , u_n\}$.

    Then $u_1^\prime + u_1 = \sum_{i=2}^n b_iu_i$ with $b_i \geq 0$ and
    $\nu_{u_1^\prime} + \nu_{u_1} > \sum_{i=2}^n b_i\nu_{u_i}$.
\end{lem}
\begin{proof}
    The proof is an induction on the rank $n=|S|$ and the length $\ell(w)$.

    If $|S|=1$ then the result is clear, so assume that $S=\{ s_1, \ldots, s_n\}$ with $n > 1$ and
    $\ell(w)=1$. Assume without loss of generality that $w=s_1$, and since $w$ covers $w'$, $w'=e$.
    If $w$ is initial for~$c$ then
    we are done by Lemma~\ref{lem:s_initial_for_c}. So assume that $w$ is not initial for~$c$.
    Then we have $u_i=v_{s_i}$, for $2\leq i\leq n$, $u'_1=v_{s_1}$, and $u_1=u$ for 
    some $u \in \mathcal R$.  Thus the maximal cones~$C(e)$ and~$C(w)$ are generated 
    by $\{ v_{s_1}, v_{s_2}, \ldots , v_{s_n}\}$ and $\{ u, v_{s_2},\ldots,v_{s_n}\}$.
    For sake of definiteness, suppose $s=s_2$ is initial in~$c$. Then~$C(e)$ and~$C(w)$ 
    are both below~$H_s$.
    By Lemma~\ref{lem:cones}, we have maximal cones $C_{sc}(e)=C(e)\cap H_s$
    and $C_{sc}(w)=C(w)\cap H_s$ in the $sc$-Cambrian fan~$\F_{sc}$ of~$W_{\la s \ra}$ and these
    cones are generated by $\{ v_{s_1}, v_{s_3}, \ldots , v_{s_n}\}$ and
    $\{ u, v_{s_3}, \ldots , v_{s_n}\}$. So by induction on the rank of~$|S|$, we obtain
    the claim with $b_2=0$.

    \medskip
    For the induction, we assume that the claim is true whenever $\widetilde w$ is $\widetilde c$-sortable
    for a Coxeter group generated by $\widetilde S$ with $|\widetilde S| < |S|$ or $\widetilde w$ is a
    $c$-sortable element with $\ell(\widetilde w) < \ell(w)$.

    Assume $w, w^\prime \in W$ are $c$-sortable with $\ell(w)>1$ and $w$ covers $w^\prime$ in the
    lattice of $c$-sortable elements. Let $s \in S$ be initial in~$c$. We split into cases based 
    on the position of~$C(w)$ and~$C(w')$ relative to~$H_s$. Note that~$C(w)$ below~$H_s$
    and~$C(w^\prime)$ above~$H_s$ is not possible since~$w$ covers~$w^\prime$ in the $c$-Cambrian
    lattice.

    \textbf{Case~$\pmb 1$:} Suppose $C(w)$ and $C(w^\prime)$ are above~$H_s$.
    The ray~$\rho_s$ is strictly below~$H_s$ by~\cite[Lemma~$6.3$]{reading4},
    so $v_s \not\in \{ u^\prime_1, u_1, \ldots , u_n \}$. Moreover,
    we conclude from Lemma~\ref{lem:six} that the maximal cones $C_{scs}(sw)$ and $C_{scs}(sw^\prime)$
    in~$\F_{scs}$ are generated by $\{ s(u_1), \ldots, s(u_n)\}$ and
    $\{ s(u_1'), s(u_2), \ldots, s(u_n)\}$ since
    $\ell(sw)<\ell(w)$, $\ell(sw^{\prime})<\ell(w^\prime)$ and $w,w^\prime > s$ in the right weak order.
    We have $C_{scs}(sw)\cap C_{scs}(sw^\prime) \subset H_{sts}$ because $C(w)\cap C(w^\prime) \subset H_{t}$.
    By induction on the length, we have
    \[
      s(u_1) + s(u_1^\prime) = \sum_{i=2}^{n}b_is(u_i)
      \quad \text{ and } \quad
      \nu_{s(u_1)} + \nu_{s(u_1^\prime)} > \sum_{i=2}^{n}b_i\nu_{s(u_i)} \text{ with $b_i \geq 0$.}
    \]
    Applying~$s$ to these (in)equalities yields
    \[
      u_1 + u_1^\prime = \sum_{i=2}^{n}b_iu_i \in H_s
      \quad \text{ and } \quad
      \nu_{u_1} + \nu_{u_1^\prime} > \sum_{i=2}^{n}b_i\nu_{u_i} \text{ with $b_i \geq 0$,}
    \]
    since $\nu_{u}$ depends only on the orbit of $u$ under the action of $W$.

    \textbf{Case~$\pmb 2$:} $C(w)$ and $C(w^\prime)$ are below~$H_s$. Since~$w$ is $c$-sortable 
    and $\ell(sw)>\ell(w)$, we have that $w\in W_{\la s\ra}$, and similarly for~$w'$. The ray~$\rho_{s}$ 
    is the only ray of~$\F_c$ strictly below~$H_s$ by~\cite[Lemma~$6.3$]{reading4}, hence we may assume 
    that $u_2=v_s$. Now $\{ u_1, u_3, \ldots , u_n \}$ and $\{ u_1^\prime, u_3,\ldots , u_n \}$ generate 
    the extremal rays of maximal cones $C_{sc}(\widetilde w), C_{sc}(\widetilde w^\prime) \subset H_s$ of 
    the $sc$-Cambrian fan~$\F_{sc}$ with $\widetilde w,\widetilde w^\prime \in W_{\la s \ra}$. The claim 
    follows by induction on the rank~$|S|$.

    \textbf{Case~$\pmb 3$:} $C(w)$ is above~$H_s$ and $C(w^\prime)$ is below~$H_s$.
    Hence~$C(w)$ and~$C(w^\prime)$ are separated by~$H_s$, so we have $s=t$. Hence
    $u^\prime_1=v_s$ ($\rho_{s}$ is the only ray of~$\F_c$ below~$H_s$) and there is 
    a maximal cone~$C_{scs}(g)$ for some $scs$-sortable element~$g\in W$ which is 
    generated by the extremal rays $\zeta_s(u^\prime_1),\zeta_s(u_2),\dots,\zeta_s(u_n)$. 
    Now, observe that
    \begin{align*}
        \zeta_s(u_1)&=s(u_1), \quad\\
        \zeta_s(u_1 ^\prime)&=-u_1^\prime=-v_s, \\
        \text{and }\zeta_s(u_i)&=u_i \text{ for $2\leq i \leq n$}
    \end{align*}
    Thus the maximal cones~$C_{scs}(g)$ and $C_{scs}(sw)$ have extremal rays generated by
    $-v_s,u_2,\dots,u_n$ and    $s(u_1), u_2,\dots,u_n$. Moreover, 
    $C_{scs}(g) \cap C_{scs}(sw) \subseteq H_s$.

    We first show that $g=w$. By definition of $\F_c$, we have $sw(D)\subset C_{scs}(sw)$.
    From~$C_{scs}(sw) \cap C_{scs}(g) \subseteq H_s$ we deduce that
    $w(D)\subset C_{scs}(g)$ or equivalently
    $w \in (\pi^{scs}_\downarrow)^{-1}(g)$. Now $g > sw$ implies $h>sw$ for all
    $h \in (\pi^{scs}_\downarrow)^{-1}(g)$. But~$C(w)$ is above~$H_s$, so~$h<w$ implies
    $h \not\in (\pi^{scs}_\downarrow)^{-1}(g)$. Hence~$w$ is the minimal element of
    $[\pi^{scs}_\downarrow(g),\pi^\uparrow_{scs}(g)]$ and we have $w=g$.

    Though~$C_{scs}(w)\cap C_{scs}(sw)\subseteq H_s$,  it is not possible to apply the induction
    hypothesis immediately, since the length~$\ell(w)$ has not been
    reduced, but we claim that for any~$z \in S$ initial in $scs$ either~$C_{scs}(w)$ and~$C_{scs}(sw)$ are both
    above~$H_z$ or both below~$H_z$. From $z \in S\setminus \{s\}$
    we conclude that $v_z \in H_s$.  We know that $u_2,\dots,u_n\in H_s$
    and $u_1^\prime,u_1,s(u_1) \not\in H_s$. So $v_z \in C_{scs}(w)$
    if and only if $v_z \in C_{scs}(sw)$. Since $v_z$ is the only ray of~$\F_{scs}$ below the hyperplane~$H_z$,
    we have shown that~$C_{scs}(w)$ and~$C_{scs}(sw)$ are on the same side of~$H_z$.

    This implies that we are now in Case~$1$ or Case~$2$, so we first apply the
    argument of the relevant case to apply the induction hypothesis and conclude
    \[
      \zeta_s(u_1)+\zeta_s(u_1^\prime)
            = \sum_{i=2}^n b_i u_i\in H_s \quad\text{with $b_i\geq 0$}.
    \]
    Therefore
    \begin{align*}
      s(u_1+u_1^\prime)
                        &= \left( \zeta_s(u_1) + \zeta_s(u_1^\prime) \right)
                                                         + \left( v_s + s(v_s) \right) \in H_s
    \end{align*}
    and $u_1 + u_1^\prime= \sum_{i=2}^n b_i u_i \in H_t$ with $b_i \geq 0$, since $s=t$.

    \medskip
    It remains to prove $\nu_{u_1}+\nu_{u_1^\prime} > \sum_{i=2}^{n}b_i\nu_{u_i}$.
    By Lemma~\ref{lem:equivalent_condition} it is sufficient to show that $\scalprod{x(C(w'))-x(C(w))}{u'_1} > 0$.

    Recall that  $t=s\in S$. Pick a maximal chain in the $c$-Cambrian lattice
    \[
      y_0 \lessdot y_1 \lessdot \ldots \lessdot y_p
    \]
    with $y_0=s$ and $y_p=w$. Then $s \leq y_i$ for $0 \leq i \leq p$,
    so~$C(y_i)$ is above~$H_s$ for $0 \leq i \leq p$. So
    for the pair $\widetilde w^\prime = y_{i-1}$
    and $\widetilde w = y_i$ we have $z_i :=
    x(C(y_i))-x(C(y_{i-1}))=\mu_i\beta_i$ with $\mu_i>0$ and
    $\beta_i\in \Phi^-$ by Lemma~\ref{lem:equivalent_condition} and Case~$1$ above.
    Now $\scalprod{\beta_i}{v_s}$ is the coefficient of the simple root~$\alpha_s$ 
    in the simple root expansion of $\beta_i$.  Since $\beta_i$ is a negative root,
    $\scalprod{\beta_i}{v_s}\leq 0$. In particular we have
    \[
      \scalprod{x(C(y_{i-1}))}{v_s} \geq \scalprod{x(C(y_{i-1}))}{v_s}+\scalprod{z_i}{v_s}
                                    = \scalprod{x(C(y_{i}))}{v_s}
    \]
    for $1\leq i \leq p$. Hence
    \[
         \scalprod{x(C(e))}{v_s}
                  > \scalprod{x(C(s))}{v_s}
                  \geq \scalprod{x(C(y_2))}{v_s}\geq \dots\geq \scalprod{x(C(w))}{v_s},
    \]
    where the first inequality is Lemma~\ref{lem:s_initial_for_c}. As $u_1^\prime=v_s$ we have
    \[
      \scalprod{x(C(e))}{v_s} = \nu_{v_s}=\nu_{u_1'}=\scalprod{x(C(w^\prime))}{v_s}
    \]
    Thus $\scalprod{x(C(w'))-x(C(w))}{u_1'}>0$.
 \end{proof}

\subsection{On integer coordinates}\label{subsec:integer_coords}

Suppose that~$W$ is a Weyl group and that the root system~$\Phi$
for~$W$ is crystallographic, that is, for any two
roots~$\alpha,\beta\in\Phi$ we have $s_\alpha(\beta)=\beta +
\lambda\alpha$ for some~$\lambda\in\mathbb Z$. The simple
roots~$\Delta$ span the lattice~ $L$ and the  fundamental weights
$v_s$, $s\in S$, span a lattice~$L^*$ which is dual to~$L$.
For~$\beta\in L$ and~$v\in L^*$ we have $\la \beta,v\ra\in\Z$. In
fact,~$\beta\in L$ if and only if $\la \beta,v\ra\in \Z$ for
all~$v\in L^*$. For each ray~$\rho\in\F_c$, we denote by~$v_\rho \in
L^*$ the lattice point closest to the origin.

\begin{lem}
   Let~$\Phi$ be a crystallographic root system for the Weyl group~$W$ and~$c$ a
   Coxeter element of~$W$. The set~$\{v_\rho\mid \rho$ an extremal ray of $C\}$ forms a basis
   of~$L^*$ for each maximal cone~$C\in\F_c$.
\end{lem}

\begin{proof}
   Let~$C=C(w)$ denote the maximal cone of~$\F_c$ for some $c$-sortable~$w\in W$. The
   proof is by induction on~$\ell(w)$ and the rank of~$W$. Let~$s$ be initial in $c$ and by
   Lemma~\ref{lem:w_c_sortable} we have to distinguish two cases,

   Suppose that~$\ell(sw)<\ell(w)$. Then~$sw$ is $scs$-sortable
   and~$C(w)=s\left(C_{scs}(sw)\right)$. Since the simple reflection~$s$ preserves the lattice,
   the result follows by induction.

   Suppose on the other hand that $\ell(sw)>\ell(w)$. Then the cone~$C(w)$ lies below the
   hyperplane~$H_s$ and $w\in W_{\la s \ra}$ is~$sc$-sortable. Let~$C_{\la s \ra}(w)$ denote
   the maximal cone that corresponds to~$w$ in the Cambrian fan~$\F_{\la s \ra} \subset H_s$
   for~$W_{\la s \ra}$. Then~$C_{\la s \ra}(w) = C(w)\cap H_s$ by Lemma~\ref{lem:cones}.
   The induction hypothesis implies that the extremal rays of~$C_{\la s \ra}(w)$ form a basis
   for the lattice~$L^*_{\la s \ra} \subset H_s$ and~$\rho_s$ is the unique extremal ray of~$C(w)$
   not contained in~$H_s$ by Lemma~\ref{lem:cones}.  Since the fundamental weights $v_t$, $t\in S$,
   span~$L^*$ it follows that~$L^*$ is spanned by~$v_s$ and $L^*_{\la s \ra}=L^*\cap H_s$.
   Hence, the extremal rays of~$C(w)$ span~$L^*$.
\end{proof}

\begin{thm}\label{thm:IntegerCoordinates}
   Let~$\Phi$ be a crystallographic root system for the Weyl group~$W$ and~$c$ a Coxeter element
   of~$W$. Suppose that~$\pmb a \in L$. Then the vertex sets~$V(Perm^{\pmb a}(W))$ and $V(\Ass^{\pmb a}_c(W))$
   are   contained in~$L$.
\end{thm}
\begin{proof}
   The result for the permutahedron is classic.
   Let~$w\in W$ be $c$-sortable,~$x(w)$ be the vertex of $\Ass^{\pmb a}_c (W)$ contained in the
   maximal cone~$C(w)\in \F_c$, and~$\rho_i$, $1 \leq i \leq n$ be the extremal rays of~$C(w)$.
   Denote the lattice point on~$\rho_i$ closest to the origin by~$y_i$. The point~$x(w)$ satisfies
   $\la x(w),y_i\ra=c_i$ for some integer~$c_i$ since ${\pmb a}\in L$. Because~$\{y_i\}$,
   $1 \leq i \leq n$, is a basis of~$L^*$, this set of equations for $x(w)$ has an integral
   solution. In other words, $x(w)\in L$.
\end{proof}

\section{Observations and remarks}\label{se:Remaarks}

\subsection{Recovering the $c$-cluster complex from the $c$-singletons}

\begin{figure}[b]
      \begin{center}
      \begin{minipage}{0.95\linewidth}
         \begin{center}
         \begin{overpic}
            [width=\linewidth]{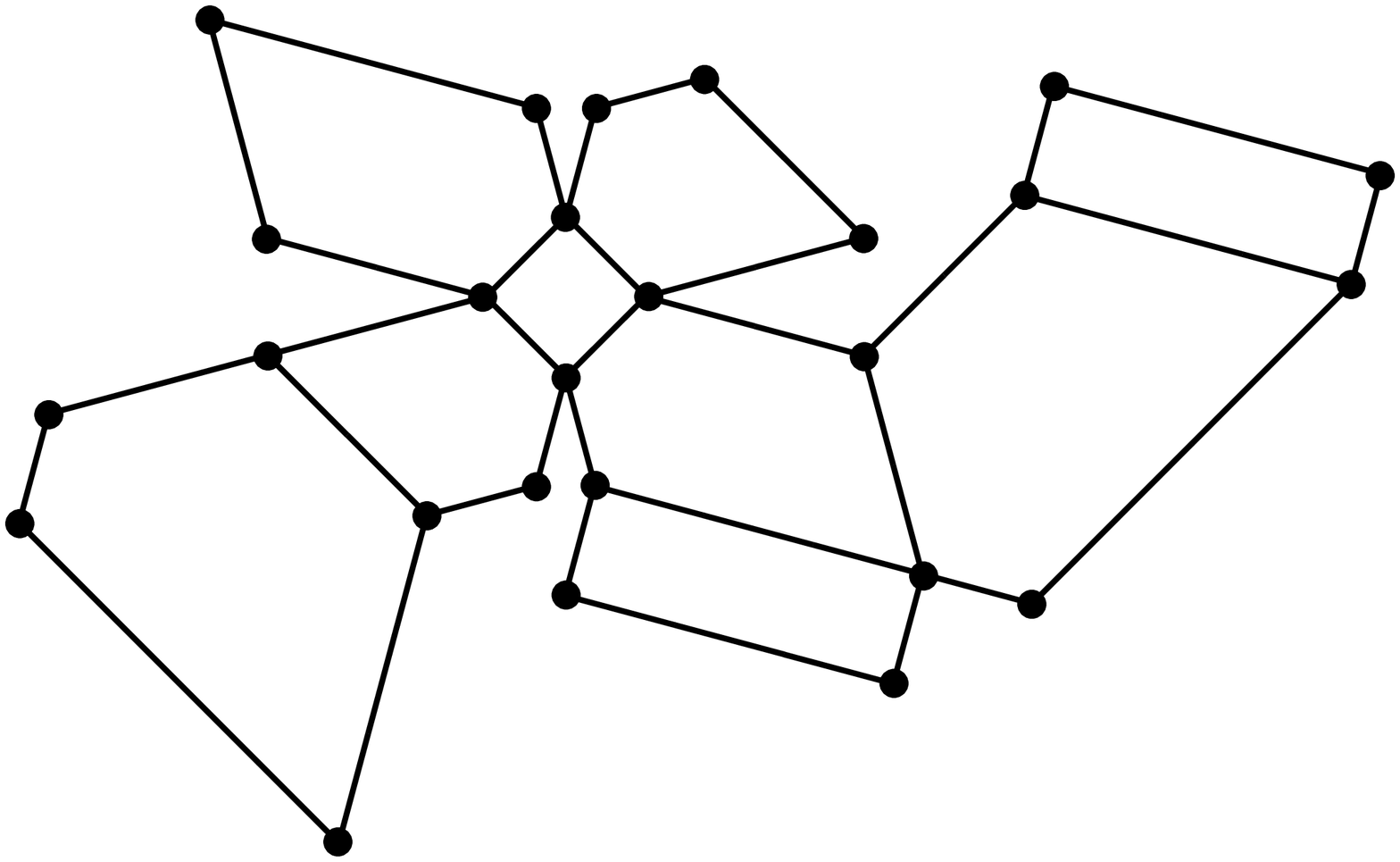}
            \put(29.5,24){\tiny $A$}
            \put(42,21.5){\tiny $A$}
            \put(36.5,28){\tiny $B$}
            \put(43.5,28){\tiny $B$}
            \put(42,32){\tiny $C$}
            \put(34.7,39.3){\tiny $D$}
            \put(44.5,35){\tiny $E$}
            \put(38,42.3){\tiny $F$}
            \put(36,46){\tiny $G$}
            \put(43.5,46.5){\tiny $G$}
            \put(69,47){\tiny $G$}
            \put(47,47.3){\tiny $H$}
            \put(65.5,40.5){\tiny $H$}
            \put(58,17){\tiny $1$}
            \put(66.5,21.5){\tiny $1$}
            \put(27,8.5){\tiny $1$}
            \put(62,23){\tiny $2$}
            \put(59,33){\tiny $3$}
            \put(55,41.2){\tiny $3$}
            \put(82.5,37){\tiny $4$}
            \put(11.5,24.5){\tiny $4$}
            \put(85,43){\tiny $5$}
            \put(12.5,28.8){\tiny $5$}
            \put(22.5,50.5){\tiny $5$}
            \put(22.5,31.5){\tiny $6$}
            \put(25,42){\tiny $6$}
            \put(77,1){\begin{overpic}[width=4cm]{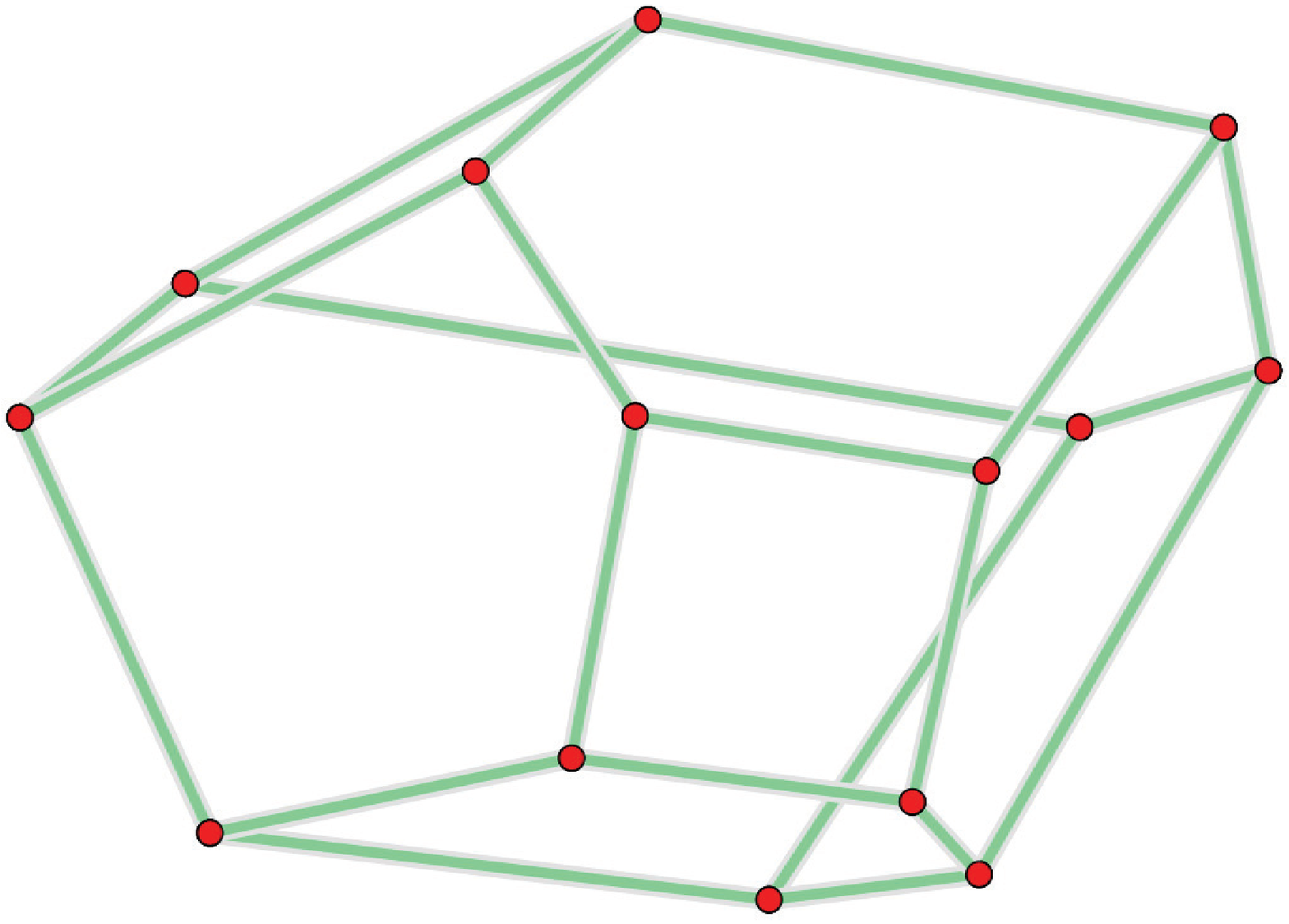}
                           \put(53,2){\tiny $A$}
                           \put(68,5){\tiny $B$}
                           \put(64.8,12.5){\tiny $C$}
                           \put(43.5,16.5){\tiny $D$}
                           \put(60.5,28.5){\tiny $E$}
                           \put(39.5,33.5){\tiny $F$}
                           \put(37.7,49.7){\tiny $G$}
                           \put(43.8,63){\tiny $H$}
                           \put(73.5,31){\tiny $1$}
                           \put(86.5,36){\tiny $2$}
                           \put(83.5,55){\tiny $3$}
                           \put(13,45){\tiny $4$}
                           \put(3,33.5){\tiny $5$}
                           \put(16,6){\tiny $6$}
                        \end{overpic}}
            \put(20,23){$-\a_{s_1}$}
            \put(43,44){\tiny $\a_{s_1}+\a_{s_2}$}
            \put(46,41){\tiny $+\a_{s_3}$}
            \put(49,31){$\a_{s_1}$}
            \put(67,33){$\a_{s_3}$}
            \put(71,44){$\a_{s_2}+\a_{s_3}$}
            \put(39.5,40){\tiny $\a_{s_1}$}
            \put(40.5,38){\tiny $+$}
            \put(39.5,36.5){\tiny $\a_{s_2}$}
            \put(49,21){$-\a_{s_2}$}
            \put(29.5,32){$-\a_{s_3}$}
            \put(29,45){$\a_{s_2}$}
         \end{overpic}
         \end{center}
         \caption[]{An unfolding of the associahedron~$\Ass^{\pmb a}_{c}(S_4)$ with $c=s_1s_2s_3$,
                    the polar of the $c$-cluster complex. The $2$-faces are labeled by replacing the
                    labels~$w(\rho_s)$ in Figure~\ref{fig:unfolded_a3_asso_123} by the almost positive
                    root~$\lr_s(w)$. }
         \label{fig:unfolded_a3_asso_123-II}
      \end{minipage}
      \end{center}
\end{figure}

It is possible to obtain polytopal realizations of the $c$-cluster
complex from the construction of generalized associahedra presented
as follows. Suppose that we are given
a $W$-permutahedron~$\Perm^{\pmb a}(W)$, a
Coxeter element~$c$, and the $c$-sorting word~$\ww$ of~$w_0$.
Then we can easily compute all $c$-singletons using the
characterization given in Theorem~\ref{thm:cSingleton}. The
associahedron~$\Ass^{\pmb a}_c(W)$ is now obtained from~$\Perm^{\pmb
a}(W)$ by keeping all admissible inequalities for the permutahedron,
that is all inequalities~$\la v,w(v_s)\ra \leq \la \pmb a,v_s\ra$
for~$c$-singleton~$w$. We label the facet~$\la v,w(v_s)\ra \leq
\la \pmb a,v_s\ra$ of~$\Ass^{\pmb a}_c(W)$ by the almost positive
root~$\lr_s(w)$ and extend this labeling to the Hasse diagram
of~$\Ass_c^{\pmb a}(W)$ as follows: if a face~$f$ is the
intersection of facets~$F_1, \ldots, F_k$ then assign~$f$ the union
of the almost positive roots assigned to~$F_1, \ldots,F_k$. By
Theorem~\ref{thm:Explif_c}, this labeling matches the labeling of
the $c$-Cambrian fan by almost positive roots given by Reading and
Speyer. Therefore, the opposite poset of this labeled Hasse diagram
is the face poset of the $c$-cluster complex because it is the face
poset of the $c$-Cambrian fan~$\F_c$. The polar of~$\Ass_c^{\pmb
a}(W)$ is therefore a polytopal realization of the $c$-cluster
complex. In particular, a set of almost positive roots is
$c$-compatible (see~\cite{reading2}) if and only if it can be obtained
as the intersection of some facets of $\Ass_c^{\pmb a}$ by the
process described above.

\medskip
We illustrate the recovery of the $c$-cluster complex for~$W=S_4$.
Figure~\ref{fig:unfolded_a3_asso_123-II} refers to the Coxeter element~$c=s_1s_2s_3$ and
Figure~\ref{fig:unfolded_a3_asso_213-II} refers to the Coxeter element~$c=s_2s_1s_3$.
We use the polar of the $c$-cluster complex for the illustration.

\begin{figure}
      \begin{center}
      \begin{minipage}{0.95\linewidth}
         \begin{center}
         \begin{overpic}
            [width=\linewidth]{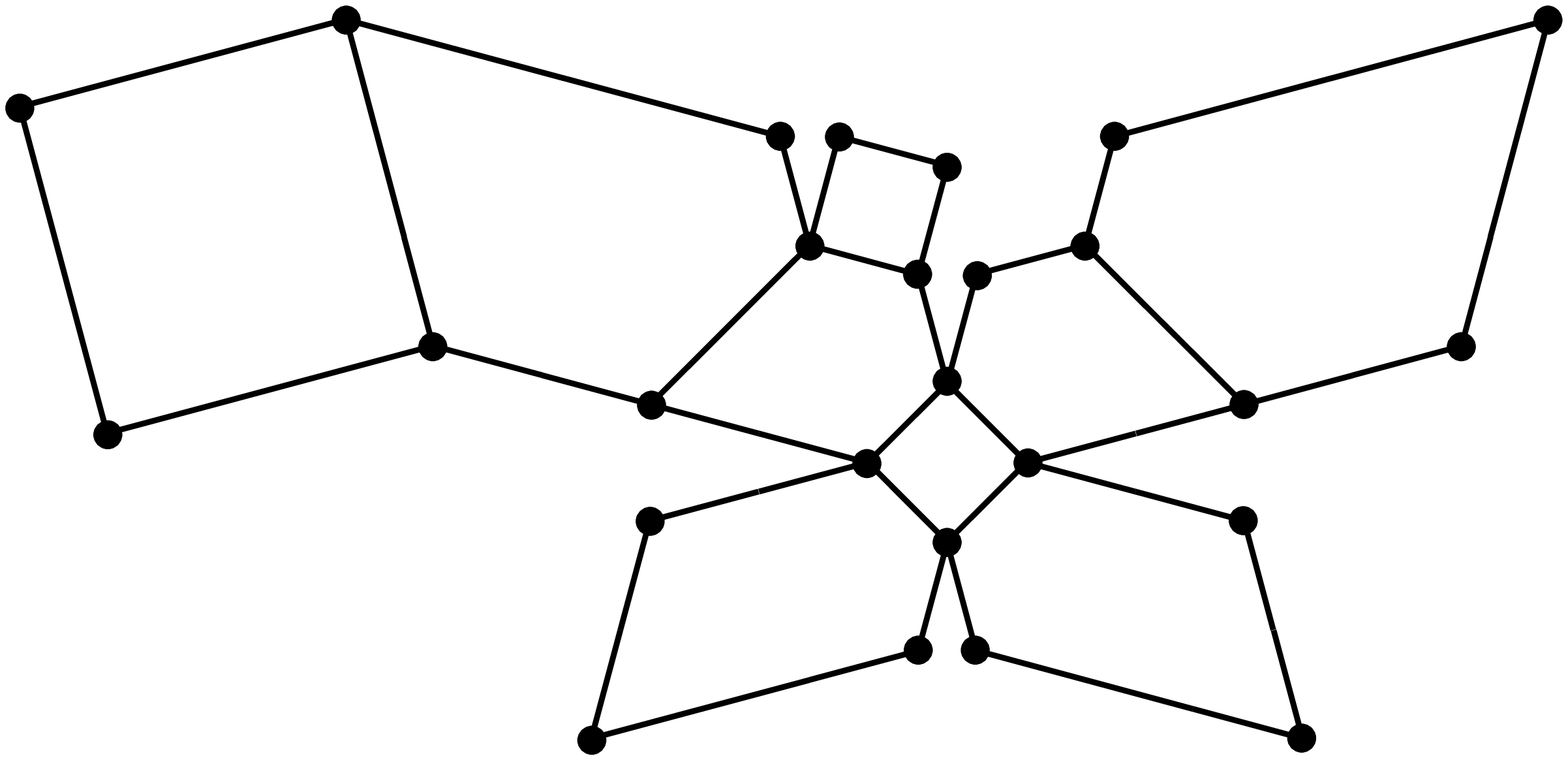}
            \put(55,18){\tiny $A$}
            \put(62,18){\tiny $A$}
            \put(11.5,31){\tiny $A$}
            \put(60.5,22){\tiny $B$}
            \put(53,25){\tiny $C$}
            \put(63.5,25){\tiny $D$}
            \put(60.5,32){\tiny $E$}
            \put(62,36.5){\tiny $F$}
            \put(55,36.5){\tiny $F$}
            \put(48,39.7){\tiny $G$}
            \put(60.5,44){\tiny $H$}
            \put(68.5,39.5){\tiny $H$}
            \put(47.5,44.5){\tiny $I$}
            \put(53.5,44){\tiny $I$}
            \put(70,44.5){\tiny $I$}
            \put(8,46.5){\tiny $1$}
            \put(77.5,13.5){\tiny $1$}
            \put(87,35.5){\tiny $1$}
            \put(74.5,22){\tiny $2$}
            \put(76.5,33){\tiny $2$}
            \put(91,50){\tiny $3$}
            \put(23.5,50){\tiny $3$}
            \put(40.5,13.5){\tiny $4$}
            \put(28,35){\tiny $4$}
            \put(43,22){\tiny $5$}
            \put(44.5,31.5){\tiny $5$}
            \put(77,1){\begin{overpic}
                            [width=3.5cm]{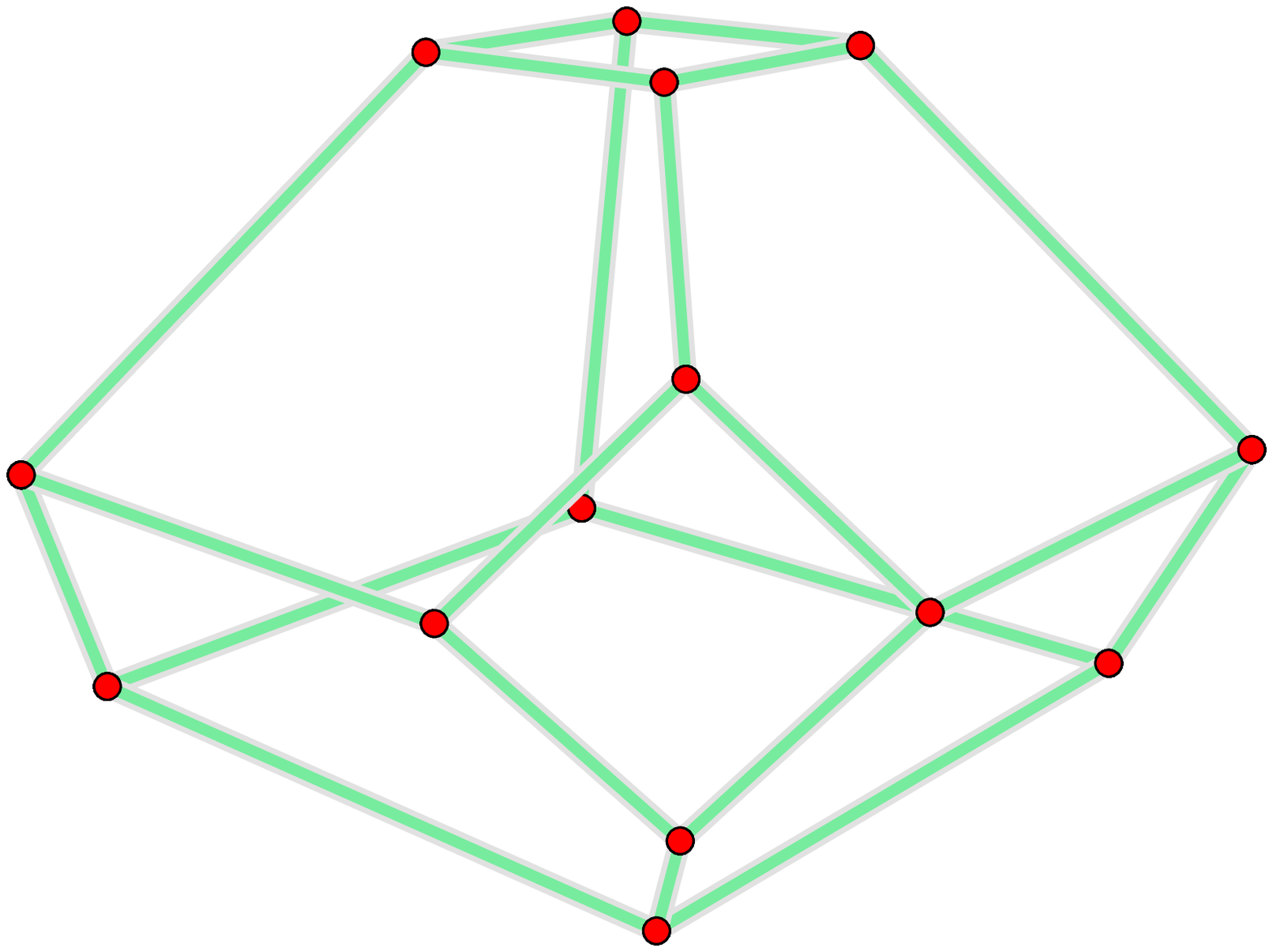}
                            \put(58,0){\tiny $A$}
                            \put(53,8){\tiny $B$}
                            \put(39,18){\tiny $C$}
                            \put(73,18){\tiny $D$}
                            \put(62,39){\tiny $E$}
                            \put(61,52){\tiny $F$}
                            \put(37,59){\tiny $G$}
                            \put(74,59){\tiny $H$}
                            \put(55,63.5){\tiny $I$}
                            \put(87,17){\tiny $1$}
                            \put(97,31){\tiny $2$}
                            \put(52,25){\tiny $3$}
                            \put(14,29){\tiny $4$}
                            \put(21,15){\tiny $5$}
            \end{overpic}}
            \put(15,41){$-\a_{s_2}$}
            \put(36,41){$\a_{s_1}$}
            \put(46,33.4){$\a_{s_1}+\a_{s_2}$}
            \put(61,34){$\a_{s_2}+\a_{s_3}$}
            \put(77,41){$\a_{s_3}$}
            \put(57.5,28){$\a_{s_2}$}
            \put(47,20){$-\a_{s_3}$}
            \put(67,20){$-\a_{s_1}$}
            \put(54,49){\tiny $\a_{s_1}+\a_{s_2}$}
            \put(58.3,47){\tiny $+\a_{s_3}$}
         \end{overpic}
         \end{center}
         \caption[]{An unfolding of the associahedron~$\Ass^{\pmb a}_{c}(S_4)$ with~$c=s_2s_1s_3$, the
                    polar of the $c$-cluster complex. The $2$-faces are labeled by replacing the
                    labels~$w(\rho_s)$ in Figure~\ref{fig:unfolded_a3_asso_213} by the almost positive
                    root~$\lr_s(w)$. }
         \label{fig:unfolded_a3_asso_213-II}
      \end{minipage}
      \end{center}
\end{figure}

First consider the Coxeter element $c=s_1s_2s_3$. The facets are labeled by almost positive
roots as indicated. The vertices correspond to clusters as follows:
{\small
\begin{alignat*}{2}
   A &= \{-\a_{s_1},-\a_{s_2},-\a_{s_3}\},
             &\quad B &= \{\a_{s_1},-\a_{s_2},-\a_{s_3}\} \\
   C &= \{\a_{s_1},\a_{s_1}+\a_{s_2},-\a_{s_3}\},
             &\quad D &= \{\a_{s_2},\a_{s_1}+\a_{s_2},-\a_{s_3}\},\\
   E &= \{\a_{s_1},\a_{s_1}+\a_{s_2},\a_{s_1}+\a_{s_2}+\a_{s_3}\},
             &\quad F &= \{\a_{s_2},\a_{s_1}+\a_{s_2},\a_{s_1}+\a_{s_2}+\a_{s_3}\},\\
   G &= \{\a_{s_2},\a_{s_2}+\a_{s_3},\a_{s_1}+\a_{s_2}+\a_{s_3}\},
             &\quad H &= \{\a_{s_3},\a_{s_2}+\a_{s_3},\a_{s_1}+\a_{s_2}+\a_{s_3}\},\\
   1 &= \{-\a_{s_1},-\a_{s_2},\a_{s_3}\},
             &\quad 2 &= \{\a_{s_1},-\a_{s_2},\a_{s_3}\},\\
   3 &= \{\a_{s_1},\a_{s_1}+\a_{s_2}+\a_{s_3},\a_{s_3}\},
             &\quad 4 &= \{-\a_{s_1},\a_{s_2}+\a_{s_3},\a_{s_3}\},\\
   5 &= \{-\a_{s_1},\a_{s_2},\a_{s_2}+\a_{s_3}\},
             &\quad 6 &= \{-\a_{s_1},\a_{s_2},-\a_{s_3}\}.
\end{alignat*}
}
Second consider the Coxeter element $c=s_2s_1s_3$. The facets are labeled by almost positive
roots as indicated and the vertices correspond to clusters as follows:
{\small
\begin{alignat*}{2}
   A &= \{-\a_{s_1},-\a_{s_2},-\a_{s_3}\},
             &\quad B &= \{-\a_{s_1},\a_{s_2},-\a_{s_3}\} \\
   C &= \{\a_{s_1}+\a_{s_2},\a_{s_2},-\a_{s_3}\},
             &\quad D &= \{-\a_{s_1},\a_{s_2},\a_{s_2}+\a_{s_3}\},\\
   E &= \{\a_{s_1}+\a_{s_2},\a_{s_2},\a_{s_2}+\a_{s_3}\},
             &\quad F &= \{\a_{s_1}+\a_{s_2},\a_{s_1}+\a_{s_2}+\a_{s_3},\a_{s_2}+\a_{s_3}\},\\
   G &= \{\a_{s_1},\a_{s_1}+\a_{s_2},\a_{s_1}+\a_{s_2}+\a_{s_3}\},
             &\quad H &= \{\a_{s_3},\a_{s_2}+\a_{s_3},\a_{s_1}+\a_{s_2}+\a_{s_3}\},\\
   I &= \{\a_{s_1},\a_{s_1}+\a_{s_2}+\a_{s_3},\a_{s_3}\},
             &\quad 1 &= \{-\a_{s_1},-\a_{s_2},\a_{s_3}\},\\
   2 &= \{-\a_{s_1},\a_{s_2}+\a_{s_3},\a_{s_3}\},
             &\quad 3 &= \{\a_{s_1},-\a_{s_2},\a_{s_3}\},\\
   4 &= \{\a_{s_1},-\a_{s_2},-\a_{s_3}\},
             &\quad 5 &= \{\a_{s_1},\a_{s_1}+\a_{s_2},-\a_{s_3}\}.
\end{alignat*}}

\subsection{A conjecture about vertex barycentres}\label{su:Dihedral}

J.-L.~Loday mentions in~\cite{loday} that F.~Chapoton observed the following: the vertex
barycentres of the permutahedron and associahedron coincide in the case of Loday's original
realization of the (classical) type~$A$ associahedron. The first two authors observed the
same phenomenon for the realizations of type~$A$ and~$B$ associahedra described in~\cite{HL}.
None of these observations have been proven so far. Checking numerous examples in GAP~\cite{gap},
we observed that the vertex barycentre of~$\Perm^{\pmb a}(W)$ and~$\Ass^{\pmb a}_c(W)$ coincide
for~$\pmb a=\sum_{s\in S} av_s$, $a>0$. The cases include types~$A_n$ ($n\leq 7$),~$B_n$
and~$D_n$ ($n\leq 5$),~$F_4$,~$H_3$,~$H_4$, and dihedral groups~$I_2(m)$. The experiments
can be summarized in the following conjecture.

\begin{conj}
   Let~$W$ be a Coxeter group and~$c\in W$ a Coxeter element. Choose a real number~$a>0$ and
   set~$\pmb a = \sum_{s\in S} av_s$ to fix a realization of the permutahedron~$\Perm^{\pmb a}(W)$.
   Then the vertex barycentres of~$\Perm^{\pmb a}(W)$ and~$\Ass^{\pmb a}_{c}(W)$ coincide.
\end{conj}

It is straightforward to prove this conjecture for a special family, the dihedral groups~$G_m$
of order~$2m$. We outline the proof.

Let $m\geq 2$ be an integer. The dihedral group~$G_m$ of order~$2m$ is the finite Coxeter group
of type~$I_2(m)$ generated by the two reflections~$s$ and~$t$ with~$st$ having order~$m$. For
any~$m$, the action of~$G_m$ on~$V=\mathbb R^2$ is essential and we identify~$\mathbb R^2$ with
the the complex numbers. If we define
\[
  v_s := \frac{1+e^{i\tfrac{\pi}{m}}}{2}\quad\textrm{ and }\quad
  v_t := \frac{1+e^{-i\tfrac{\pi}{m}}}{2}
\]
then~$G_m$ is generated by the reflections with respect to the hyperplanes spanned by~$v_s$
and~$v_t$. We choose~$\pmb a = v_s+v_t=1$ and follow our earlier notation where~$M(w)$
denotes the point obtained by the action of~$w\in G_m$ on~$\pmb a$. Then
\[
  M(w) =\begin{cases}
           e^{i\ell(w)\tfrac{\pi}{m}}   &\textrm{if } \ell(sw)<\ell(w)\\
           e^{-i\ell(w)\tfrac{\pi}{m}}  &\textrm{if } \ell(sw)>\ell(w).\\
        \end{cases}
\]
The convex hull of the points~$M(w)$, $w\in G_m$, is the permutahedron~$\Perm^{\pmb a}(G_m)$ which
is a regular $2m$-gon. It is easy to verify that the origin is the vertex barycentre
of~$\Perm^{\pmb a}(G_m)$.

We argue now for the the Coxeter element~$st$; if~$c=ts$, the reasoning is similar. The
$c$-singletons are~$e$ and all $w\in G_m$ with $\ell(sw)<\ell(w)$. The generator~$t$ is
the only $c$-sortable element which is not a $c$-singleton. Denote the intersection of the
line through~$M(e)$ and~$M(t)$ and the line through~$M(w_0)$ and~$M(sw_0)$ by~$P$. The
associahedron~$\Ass_c^{\pmb a}(G_m)$ is the convex hull of the points~$M(w)$, such that $w\in G_m$ is a
$c$-singleton, and~$P$. A straightforward computation yields
\[
  P=\frac{i\sin\left({\tfrac{\pi}{m}}\right)}{\cos\left(\tfrac{\pi}{m}\right)-1}.
\]
and it is not hard to verify that
\[
  \sum_{w\in G_m \atop \text{not $c$-singleton}}M(w) = \sum_{k=1}^{m-1}\left(e^{-i\tfrac{\pi}{m}}\right)^k=P,
\]
so the vertex barycentres of~$\Perm^{\pmb a}(G_m)$ and~$\Ass^{\pmb a}_c(G_m)$ coincide.

\subsection{Recovering the realizations of~\cite{HL} for types~$A$ and~ $B$}\label{su:Our}

\subsubsection{Type A}
Let $\mathcal B=\{e_1,\dots,e_n\}$ be the canonical basis of~$\R^n$. The symmetric group~$S_n$
acts naturally on~$\R^n$ by permutation of the coordinates. We set
\[
  \Delta := \{e_{i+1}-e_i\,|\, 1\leq i \leq n-1\}
  \qquad\textrm{and} \qquad
  \Phi^+ := \{e_j-e_i\,|\, 1\leq i<j\leq n\}.
\]
Then $\Phi=\Phi^+\cup (-\Phi^+)$ is a root system of type~$A_{n-1}$ with simple root system~$\Delta$.
Moreover, we recall that the reflection group~$S_n$ acts essentially on
\[
  V := \R[\Delta] = \left\{ x=(x_1,\dots,x_n)\in\R^n \,\left|\, \sum_{i=1}^n \right. x_i=0\right\}
    \subset \mathbb R^n.
\]

Let~$s_i$ be the simple reflection that maps the simple root $e_{i+1}-e_i$ to $e_i-e_{i+1}$.
The dual basis $\Delta^*$ of $\Delta$ is described by
\[
  v_{s_i} := \frac{i-n}{n}\sum_{k=1}^i e_k+\frac{i}{n}\sum_{k=i+1}^n e_k \in V.
\]
We choose $\pmb a := \sum_{i=1}^{n-1} v_{\tau_i}$ and have
$\pmb a = \sum_{k=1}^n \left(k - \frac{n+1}{2} \right) e_k$.

There is a bijection between Coxeter elements~$c\in S_n$ and orientations of the Coxeter graph
of~$S_n$: if $s_i$ appears before $s_{i+1}$ in a reduced expression of~$c$ then the edge
between~$s_i$ and~$s_{i+1}$ is oriented from~$s_i$ to~$s_{i+1}$. The orientation is from~$s_{i+1}$
to~$s_i$ if~$s_i$ appears after~$s_{i+1}$ in a reduced expression of~$c$. Given an oriented
Coxeter graph, we can apply the construction described earlier and obtain a
permutahedron~$\Perm^{\pmb a}(S_n)$ and an associahedron~$\Ass^{\pmb a}_c (S_n)$.

Consider the affine subspace~$\mathcal V \subset \mathbb R^n$ that is a translate of~$V$
by~$v_G = \frac{n+1}{2}\sum_{i=1}^n e_i$:
\[
  \mathcal V =\left\{ x=(x_1,\dots,x_n)\in\R^n
                   \,\left|\,
                   \sum_{i=1}^n\right. x_i=\frac{n(n+1)}{2}\right\}.
\]
Translate~$\Perm^{\pmb a}(S_n)\subset V$ by~$v_G$ to
obtain~$\Perm^{\pmb a}(S_n)+ v_G\subset \mathcal V$. The vertices
of~$\Perm^{\pmb a}(S_n) + v_G$ are the orbit of~$\pmb a +v_G=\sum_{i=1}^{n}ie_i$ under the
action of~$S_n$, in other words we have
\[
  M(w)=\sum_{i=1}^n w^{-1}(i) e_i
\]
for~$w\in S_n$. The permutahedron~$\Perm^{\pmb a}(S_n)+ v_G$ was used
in~\cite{HL} but the vertices were labeled according to~$w\mapsto w^{-1}$.

\begin{prop}
   Consider a Coxeter element~$c\in S_n$ or equivalently an orientation of the Coxeter graph and
   let~$v_G$ and~$\pmb a$ be as above. The translated associahedron~$\Ass^{\pmb a}_c (S_n)+v_G$ is
   the associahedron~$\Ass_c$ constructed in~\cite{HL}.
\end{prop}
\begin{proof}
   In~\cite[Proposition~1.3]{HL}, it was proved that the $c$-singletons are the common vertices
   of the permutahedron and the associahedron and that the normal fan of the latter is~$\F_c$.
   In other words, the realization of the associahedron in~\cite{HL} matches precisely the
   description of $\Ass_c^{\pmb a}(S_n)$ given in Corollary~\ref{cor:Intro}.
\end{proof}

\subsubsection{Type B}
Consider the simple root system of type~$B$ given by
\[
  \Delta':=\{e_{n+1}-e_n\}\cup\{e_{i+1}-e_i+e_{2n+1-i}-e_{2n-i}\,|\, 1\leq i \leq n-1\}
         \subset \mathbb R^{2n}.
\]
If we set~$V':=\mathbb R[\Delta']$ then~$V'$ is a $n$-dimensional subspace of $\mathbb R^{2n}$
which is contained in~$V$, the span of
the type $A_{2n-1}$ root system as in 4.3.1.   Denote the simple reflection that corresponds to~$e_{n+1}-e_n$
by~$s_0$ and the simple reflection that corresponds to $(e_{i+1}-e_i)+(e_{2n+1-i}-e_{2n-i})$
by $s_{n-i}$. The {\em hyperoctahedral group}~$W_n$ (or Coxeter group of type $B_n$) is
generated by these reflections.
It is easy to see that $V'=V\cap \bigcap_{i=1}^{n-1} V^B_i$ where
$V^B_i:=\{ x\in\R^{2n} \,|\, x_i+x_{2n+1-i}=0\}$. In particular we have $\pmb a \in V'$.

The claim that~$\pmb a$ is in the open cone spanned by the fundamental weights of~$\Delta'$
follows from the fact that the scalar product of~$\pmb a$ with any element of~$\Delta'$ is
strictly positive.

A Coxeter element~$c\in W_n$ is related to an orientation of the Coxeter graph of~$W_n$
as in type~$A$: If $s_i$ appears before (resp. after)~$s_{i+1}$ in a reduced expression
of~$c$ then the edge between~$s_i$ and~$s_{i+1}$ is oriented from~$s_i$ to~$s_{i+1}$ (resp.
form~$s_{i+1}$ to~$s_i$). A Coxeter element or an orientation of the Coxeter
graph yields therefore a permutahedron~$\Perm^{\pmb a}(W_n)$ as described in
Section~\ref{subsection:permutahedra} and an associahedron $\Ass^{\pmb a}_c(W_n)$. The
orientation of the Coxeter graph of~$W_n$ determines a symmetric orientation of the Coxeter
graph of~$S_{2n}$ (that is, of type $A_{2n-1}$), and thus a Coxeter element~$\tilde c$ of~$S_{2n}$ and we
have
\[
  \Perm^{\pmb a}(W_n)=\Perm^{\pmb a}(S_{2n})\cap V' 
  \quad\text{and}\quad 
  \Ass^{\pmb a}_c(W_n)=\Ass^{\pmb a}_{\tilde c}(S_{2n})\cap V'.
\]

\medskip
\noindent
The following proposition is a direct consequence from the construction in~\cite{HL}.

\begin{prop}
   Consider a Coxeter element~$c\in W_n$ or equivalently an orientation of the Coxeter graph.
   Let~$v_G$ and~$\pmb a$ as above for type~$A$. The translated
   associahedron~$\Ass^{\pmb a}_c (W_n)+v_G$ is the cyclohedron constructed in~\cite{HL}
   that corresponds to the orientation of the Coxeter graph determined by~$c$.
\end{prop}

\subsection*{Acknowledgements} The authors thank Nathan Reading and David Speyer for
having made their results in~\cite{reading4} available to us during
their work and for their helpful remarks. In particular, we thank
David Speyer for having conjectured Theorem~\ref{thm:general_result} while
reading an earlier version of this article.

The authors wish to thank the Fields Institute, where much of this
work was done, for its hospitality.


\end{document}